\documentclass[review,hidelinks,onefignum,onetabnum]{siamart220329}

\usepackage{lipsum}
\usepackage{amsfonts}
\usepackage{graphicx}
\usepackage{epstopdf}
\usepackage{algorithmic}
\ifpdf
  \DeclareGraphicsExtensions{.eps,.pdf,.png,.jpg}
\else
  \DeclareGraphicsExtensions{.eps}
\fi
\usepackage{amssymb}
\usepackage{tikz}

\DeclareMathOperator{\trace}{tr}

\newcommand{\bfxi}[0]{\boldsymbol{\xi}}

\newcommand{\khat}[0]{\widehat{K}}
\newcommand{\keq}[0]{K_\triangle}
\newcommand{\krec}[0]{\khat}

\newcommand{\tri}[0]{\mathcal{T}}
\newcommand{\trieq}[0]{\mathcal{T}_{\triangle}}
\newcommand{\trirec}[0]{\mathcal{T}_{\khat}}

\newcommand{\deggeo}[0]{{q}}

\newcommand{\pdeggeo}[0]{\mathcal{P}^\deggeo}
\newcommand{\pgeo}[0]{\mathcal{P}^\deggeo}
\newcommand{\pun}[0]{\mathcal{P}^1}
\newcommand{\pdeux}[0]{\mathcal{P}^2}

\newcommand{\Feq}[0]{F_\triangle} 
\newcommand{\jeq}[0]{J_\triangle}
\newcommand{\FeqK}[0]{F_{\triangle \to K}}
\newcommand{\jeqK}[0]{J_{\triangle \to K}}

\newcommand{\bfe}[0]{\boldsymbol{e}}

\newcommand{\bfp}[0]{{\boldsymbol{p}}}

\newcommand{\bfx}[0]{{\boldsymbol{x}}}
\newcommand{\bfy}[0]{{\boldsymbol{y}}}

\newcommand{\ecal}[0]{\mathcal{E}}
\newcommand{\fcal}[0]{\mathcal{F}}

\newcommand{\hcal}[0]{\mathcal{H}}

\newcommand{\kcal}[0]{\mathcal{K}}

\newcommand{\mcal}[0]{\mathcal{M}}
\newcommand{\mcalp}[0]{\mathcal{M}_{\bfp}}
\newcommand{\ncal}[0]{\mathcal{N}}

\newcommand{\tcal}[0]{\mathcal{T}}

\newcommand{\R}[0]{\mathbb{R}^2}
\newcommand{\Run}[0]{\mathbb{R}}
\newcommand{\Rn}[0]{\mathbb{R}^n}

\newcommand{\Rtrois}[0]{\mathbb{R}^3}

\newcommand{\eucl}[0]{\overline{\mcal}}

\newcommand{\mcalmud}[0]{\mathcal{M}^{-1/2}}

\newsiamremark{remark}{Remark}
\newsiamremark{hypothesis}{Hypothesis}
\crefname{hypothesis}{Hypothesis}{Hypotheses}
\newsiamthm{claim}{Claim}

\headers{Isometric simplices for high-order mesh adaptation}{A. Bawin, A. Garon, and J.-F. Remacle}

\title{Isometric simplices for high-order anisotropic mesh adaptation. Part I: Definition and existence of isometric triangulations.
\thanks{Submitted to the editors DATE.
\funding{This work was funded by XXX.}}}

\author{Arthur Bawin\thanks{Institute of Mechanics, Materials and Civil Engineering (iMMC), UCLouvain, Belgium}
\and André Garon\thanks{Département de Génie Mécanique, Polytechnique Montréal, Montréal QC, Canada (\email{arthur.bawin@uclouvain.be}, \email{andre.garon@polymtl.ca}, \email{jean-francois.remacle@uclouvain.be})}
\and Jean-François Remacle\footnotemark[2]}

\usepackage{amsopn}

\ifpdf
\hypersetup{
  pdftitle={An Example Article},
  pdfauthor={A. Bawin, A. Garon, and J.-F. Remacle}
}
\fi

\begin{document}
\nolinenumbers

\maketitle

\begin{abstract}
  Anisotropic mesh adaptation with Riemannian metrics has proven effective for generating straight-sided meshes with anisotropy induced by the geometry of interest and/or the resolved physics.
  Within the continuous mesh framework, anisotropic meshes are thought of as discrete counterparts to Riemannian metrics.
  Ideal, or \emph{unit}, simplicial meshes consist only of simplices whose edges exhibit unit or quasi-unit length with respect to a given Riemannian metric.
  Recently, mesh adaptation with high-order (i.e., curved) elements has grown in popularity in the meshing community,
  as the additional flexibility of high-order elements can further reduce the approximation error.
  However, a complete and competitive methodology for anisotropic and high-order mesh adaptation is not yet available.
  The goal of this paper is to address a key aspect of metric-based high-order mesh adaptation,
  namely, the adequacy between a Riemannian metric and high-order simplices.
  This is done by extending the notions of \emph{unit simplices} and \emph{unit meshes},
  central to the continuous mesh framework,
  to high-order elements.
  The existing definitions of a unit simplex are reviewed, then a broader definition involving Riemannian isometries is introduced to handle curved and high-order simplices.
  Similarly, the notion of \emph{quasi-unitness} is extended to curved simplices to tackle the practical generation of high-order meshes.
  Proofs of concept for unit and (quasi-)isometric meshes are presented in two dimensions.

\end{abstract}

\begin{keywords}
Anisotropic mesh adaptation, high-order mesh, continuous mesh, Riemannian metric, curvilinear mesh
\end{keywords}

\begin{MSCcodes}
65L50, 65N50, 53Z05
\end{MSCcodes}

\section{Introduction}

Anisotropic mesh adaptation has proven effective to provide high-quality numerical results for a reduced computational cost.
Anisotropic meshes tailored for applications with geometry- or physics-induced anisotropy (e.g., shocks) can exhibit orders of magnitude fewer degrees of freedom compared to uniform or isotropic meshes, achieving equivalent levels of accuracy.
Although the majority of studies on anisotropic mesh adaptation are concerned with linear (i.e., straight-sided) meshes,
significant research in the past decade has been dedicated to adaptivity with high-order (i.e., curved) meshes.
High-order mesh adaptation takes advantage of the additional degrees of freedom by placing curved elements not only near the CAD boundaries,
but also inside the computational domain \cite{zhang2018curvilinear, aparicio2019defining, zhang2021generation, bawin2023optimally}.
This way, the CAD model and the resolved fields can each be represented with better accuracy.

In \emph{metric-based} anisotropic mesh adaptation and, particularly within the so-called \emph{continuous mesh framework},
the process is driven by a Riemannian metric. This metric translates an error indicator into an anisotropic size field,
which is in turn converted into a discrete mesh.
Straight-sided simplicial meshes
are typically generated
by constructing edges of quasi-unit length with respect to this metric.
The ideal, or \emph{unit}, simplices \cite{loseille2008adaptation, loseille2011continuous, loseille2011continuous2}
are thus equilateral in this metric.
This edge-based criterion facilitates the generation of metric-conforming meshes,
by splitting and collapsing mesh edges to meet the specified length requirements.
However, for curved elements, an edge-based characterization of unitness is not restrictive enough.
Prescribing edge lengths alone fails to guarantee that essential invariants, such
as the area and inner angles of the simplex, or the energy of its edges, are shared among unit simplices.
To address this limitation, unitness has been redefined using the Jacobian matrix of the transformation
from the reference finite element simplex of a given simplex \cite{rocheryThesis, rochery2024generalized}.
This approach is consistent with the former definition for linear mesh elements,
and introduces on high-order simplices a pointwise constraint to ensure conformity with the ideal equilateral simplex.
However, enforcing this criterion in practice is
computationally expensive, as the edge-based condition becomes a pointwise relation to be satisfied in the whole mesh.

In this paper,
unit simplices are defined as the images of an ideal simplex by Riemannian isometries.
This definition generalizes existing notions of unitness and applies to both linear and high-order simplices.
Riemannian isometries, as metric isometries, preserve lengths, angles and areas,
but also the metric tensor between two manifolds.
Consequently,
pointwise conformity between the ideal simplex and its image is characterized by the Jacobian matrix of the isometry in coordinates,
recovering the definition from \cite{rocheryThesis, rochery2024generalized}.
Since Riemannian isometries map geodesic curves to geodesic curves,
high-order isometric simplices naturally have geodesic edges in the physical space.
This property yields a practical edge-based criterion for generating curved meshes tailored to an input metric.

The paper is organized as follows.
The adopted notations are detailed in \Cref{sec:simplices},
while the essentials of metric-based anisotropic mesh adaptation are recalled in \Cref{sec:metrics}.
\Cref{sec:prop_unit_simplices} reviews the current definitions of unit simplices.
\Cref{sec:isometric_simplices} presents the main contribution of this work:
unit simplices are extended to high-order simplices through Riemannian isometries,
and the existence of isometric triangulations is discussed.
Two definitions of quasi-unitness are proposed for the practical generation of high-order, metric-conforming triangulations.
Finally, \Cref{sec:examples} provides examples of triangulations that are isometric and quasi-isometric to a regular tiling of the plane.

\section{Reference, ideal and physical simplices}
\label{sec:simplices}

We introduce here the simplices used throughout the paper:
the \emph{reference (right) simplex}, typically the support for finite element computations;
the \emph{ideal simplex}, which serves as the target for most mesh adaptation methods;
and the arbitrary \emph{physical simplex}, which represents a generic mesh element.

The \emph{reference simplex} $\khat$ is defined in dimension $n$ using the reference coordinates $\bfxi = (\xi^1, \ldots, \xi^n)$ by the set:
\begin{equation}\label{eq:ref_simplex}
  \khat \triangleq \left\{ \bfxi \in [0,1]^n ~ \biggr| ~ \sum_{i = 1}^n \xi^i \leq 1 \right\}.
\end{equation}
The reference simplex in one dimension is the interval $[0,1]$ (although it is also often set to be $[-1,1]$), the right triangle with vertices $(0,0)$, $(1,0)$, $(0,1)$ in two dimensions,
and the trirectangular tetrahedron with vertices $(0,0,0)$, $(1,0,0)$, $(0,1,0)$, $(0,0,1)$ in three dimensions.

Most anisotropic mesh adaptation methods and meshing libraries (see for instance \cite{dobrzynski2012mmg3d}) aim to generate equilateral simplices in the space deformed by a Riemannian metric.
We therefore define the \emph{ideal simplex} $\keq$ as an equilateral simplex with edges of unit Euclidean length.
In one dimension, the ideal simplex is a unit interval and coincides with the reference simplex.
In two dimensions, $\keq$ can be represented by the equilateral triangle with vertices $(0,0), (1,0), (1/2,\sqrt{3}/2)$.
It is standard to have both the reference and ideal simplices live in a so-called reference space $\widehat{\Omega}$,
which we identify to a subset of $\R$ equipped with the Euclidean metric.

Lastly, we call \emph{physical simplex} an element $K$ of a simplicial mesh of a computational domain $\Omega \subset \mathbb{R}^n$.
In this paper, we consider only physical simplices described by a polynomial reference-to-physical transformation
$F_K : \khat \to K$.
Lagrange and Bézier bases are commonly used to describe this transformation:
the Lagrange control points correspond to the mesh vertices $\bfx_j$, while the Bézier control points
are not required to lie inside the simplex or on its boundary.
We choose here to represent the transformation $F_K$ with Lagrange polynomials $\phi_j(\bfxi)$, where $1 \leq j \leq N_q$ and
$$N_q = \frac{(q+n)!}{q!n!}$$ is the dimension of the basis of polynomials of degree $q$ in dimension $n$.
A reference-to-physical transformation of degree $q$ thus writes for $\bfxi \in \khat$:
\begin{equation}\label{eq:reftophys}
  F_K \triangleq F_{\khat \to K} : \khat \to K: \bfxi \mapsto F_K(\bfxi) = \sum_{j = 1}^{N_q} \phi_j(\bfxi)\bfx_j,
\end{equation}
and we call $\pdeggeo$ \emph{simplex} the (possibly curved) simplex $K = F_K(\khat)$.
The Jacobian matrix of $F_K$ is the matrix $J_K$ defined by:
\begin{equation}
  J_{K,ij} = \frac{\partial F_K^i}{\partial \xi^j}.
\end{equation}
Although this matrix is $n \times d$ in general, with $d$ the dimension of the simplex,
we restrict our attention to simplices whose dimension matches that of their ambient space,
such as triangles in $\R$, for which $J_K$ is an $n \times n$ matrix.
The linear transformation from the reference simplex $\khat$ to the ideal simplex $\keq$ is noted $\Feq$,
and its constant Jacobian matrix is noted $\jeq$, as shown in \Cref{fig:ref_to_phys}.
Similarly,
the transformation and Jacobian matrix from $\keq$ to the physical simplex $K$ are noted $\FeqK$ and $\jeqK$,
respectively.

\begin{figure}[tbhp]
  \centering
  \includegraphics[width=0.8\linewidth]{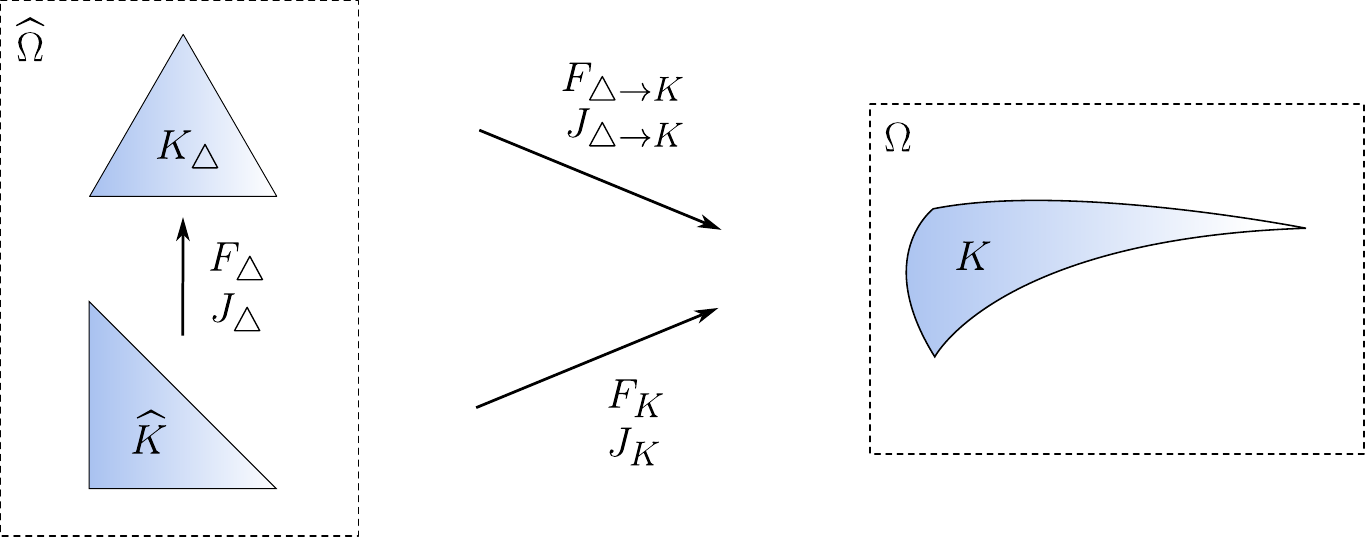}
  \caption{Notation for the transformation between the reference, ideal and physical simplices.}
  \label{fig:ref_to_phys}
\end{figure}

\section{Riemannian metrics and metric-based mesh adaptation}
\label{sec:metrics}

This paper builds on the continuous mesh framework \cite{loseille2008adaptation, loseille2011continuous, loseille2011continuous2},
which uses Riemannian metrics to define angles and lengths for mesh adaptation.
Here, we briefly recall key properties of Riemannian metrics,
namely, how they modify the computation of angles, lengths and volumes,
as well as their role within the continuous mesh framework.
Isometries, which are needed to define curved unit simplices in \Cref{sec:isometric_simplices}, are also recalled.
For a more complete treatment of Riemannian metrics and Riemannian manifolds, we refer the reader to e.g. \cite{lee2012smooth, lee2018introduction}.

Let $M$ be a smooth manifold of dimension $n$ (also called a smooth $n$-manifold) and let $x^i$ denote coordinates on $M$.
A (Riemannian) metric on $M$ is a smooth covariant 2-tensor field $\mcal = \mcal_{ij} dx^idx^j$ that is symmetric and positive-definite everywhere on $M$. In other words, it is a field of smoothly varying bilinear forms on $M$.
The matrix of components of $\mcal$ is noted $[\mcal] \triangleq \mcal_{ij}$, and is also symmetric and positive-definite.
The pair $(M,\mcal)$ of the smooth manifold $M$ equipped with the metric $\mcal$ forms a Riemannian manifold.
In the expression of $\mcal$, the coordinate differentials $dx^k$ are linear functionals taking a tangent vector and returning its $k$-th component.
At a point $\bfp \in M$,
the metric tensor $\mcalp$ takes two tangent vectors $u, v$ in the tangent space to $M$ at $\bfp$, denoted $T_\bfp M$,
and returns their inner product $\langle .,.\rangle_{\mcalp}$:
\begin{equation}
  \mcalp(u,v) = \mcal_{\bfp,ij}dx^i(u)dx^j(v) = \mcal_{\bfp,ij} u^i v^j = \langle u,v\rangle_{\mcalp}.
\end{equation}
Letting $u,v$ also denote the column vectors of their coordinates,
the inner product of $u$ and $v$ with respect to $\mcalp$ is given by the matrix product $u^T[\mcalp]v$.
The simplest example of Riemannian metric is the Euclidean metric $\eucl$,
whose components in standard coordinates are $\eucl_{ij} = \delta_{ij}$.
The associated inner product is the canonical inner product of $\Rn$, as $\langle u,v\rangle_{\eucl} = \delta_{ij} u^i v^j = \sum_{i= 1}^n u^iv^i$.

Non-Euclidean metrics modify the computation of the following useful quantities in mesh generation.
The \emph{norm} of a tangent vector with respect to the metric tensor $\mcalp$ is defined by $\Vert u \Vert_{\mcalp} \triangleq \sqrt{\langle u,u\rangle_{\mcalp}}$,
and the \emph{angle} formed by the vectors $u,v$ with respect to  $\mcalp$ is the unique real number $\theta \in [0,\pi]$ such that:
\begin{equation}\label{eq:cos_angle}
  \cos \theta = \frac{\langle u,v\rangle_{\mcalp}}{\Vert u \Vert_{\mcalp}\Vert v \Vert_{\mcalp}}.
\end{equation}
The \emph{length} with respect to $\mcal$ of a curve $\gamma$ on $M$ parameterized by $t \in [a,b]$ is:
\begin{equation}
  \ell_{\mcal}(\gamma) = \int_a^b \Vert \gamma'(t) \Vert_{\mcal}\,dt = \int_a^b \sqrt{\left \langle \gamma'(t), \gamma'(t) \right\rangle_{\mcal}} \,dt.
  \label{eq:length_in_metric}
\end{equation}
For arbitrary curves and metrics,
lengths are typically evaluated using numerical quadrature.
In anisotropic meshing, for instance,
the curves of interest are straight lines in $\Rn$, such as $\gamma = \bfx_0\bfx_1$,
simplifying the length computation to:
\begin{equation}
  \begin{aligned}
  \ell_{\mcal}(\bfx_0\bfx_1) &= \int_0^1 \Vert \bfx_1-\bfx_0 \Vert_{\mcal}\,dt\\
  &= \int_0^1 \sqrt{\left\langle \bfx_1-\bfx_0,\bfx_1-\bfx_0\right\rangle_{\mcal}} \,dt\\
  &= \int_0^1 \sqrt{(\bfx_1-\bfx_0)^T [\mcal]_{\left(\bfx_0 + t(\bfx_1-\bfx_0)\right)}(\bfx_1-\bfx_0)} \,dt.
\end{aligned}
\end{equation}
This integral is evaluated numerically to compute the length of the segment $\bfx_0\bfx_1$ with respect to the varying metric $\mcal$.
If the metric is constant on $M$, then the integral further simplifies to:
\begin{equation}
  \ell_{\mcal}(\bfx_0\bfx_1) = \int_0^1 \sqrt{(\bfx_1-\bfx_0)^T [\mcal] (\bfx_1-\bfx_0)} \,dt = \sqrt{(\bfx_1-\bfx_0)^T [\mcal] (\bfx_1-\bfx_0)}.
\end{equation}
In the continuous mesh framework,
a simplex is said to be \emph{unit} if all of its edges
have unit length with respect to the Riemannian metric.
When the metric is chosen to minimize a given error functional,
unit simplices with respect to this metric are endowed with the same error level,
yielding error-minimizing meshes.
The goal of metric-based mesh adaptation is thus to generate
meshes consisting exclusively of such unit simplices, called \emph{unit meshes}.
As unit meshes do not generally exist for arbitrary Riemannian metrics,
a relaxed definition of unitness is adopted:
\emph{quasi-unit} simplices have edges of length in $[1/\sqrt{2},\sqrt{2}]$, to allow for the practical generation of quasi-unit meshes \cite{loseille2008adaptation}.

The \emph{energy} of a curve $\gamma$ is defined analogously to the length by:
\begin{equation}
  E_{\mcal}(\gamma) = \frac{1}{2}\int_a^b \Vert \gamma'(t) \Vert_{\mcal}^2\,dt = \frac{1}{2}\int_a^b \left \langle \gamma'(t), \gamma'(t) \right\rangle_{\mcal} \,dt.
  \label{eq:energy_in_metric}
\end{equation}

The \emph{volume} of a simplex $K$ in a Riemannian manifold $(M,\mcal)$ can be computed on the reference simplex as:
\begin{equation}
  |K|_\mcal = \int_{\khat} \sqrt{\det [\mcal]} (\det J_F)\,d\bfxi.
\end{equation}
For example, in the Euclidean space $(\Rn,\eucl)$,
the volume of an arbitrary simplex $K$ can be computed on the reference element as $|K|_{\eucl} = \int_{\khat} \det J_F \,d\bfxi$.
On the manifold $(\Rn,\mcal)$ with constant metric $\mcal$, the volume of $K$ satisfies $|K|_\mcal = \sqrt{\det [\mcal]} |K|_{\eucl}$.

We conclude this section by recalling the notions of \emph{pullback metric} and \emph{isometries},
which are needed to define isometric unit elements.
If $(M,\mcal)$ and $(N,\ncal)$ are Riemannian manifolds and $F : M \to N$ is a smooth map between them,
then the \emph{pullback} on $M$ of the metric $\ncal$, noted $F^*\ncal$, is defined by the pointwise relation at any $\bfp \in M$:
\begin{equation}
  (F^*\ncal)_\bfp(u, v) = \ncal_{F(\bfp)}(dF_\bfp(u),  dF_\bfp(v)),
\end{equation}
where $u, v \in T_\bfp M$ are tangent vectors and $dF_\bfp(u), dF_\bfp(v)$ are their images in $T_{F(\bfp)}N$ by the action of the \emph{differential} $dF_\bfp : T_\bfp M \to T_{F(\bfp)}N$,
which is a linear map between tangent spaces.
The pullback metric thus evaluates the original metric at the images of its arguments.
Given coordinates $x^i$ on $M$ and $y^j$ on $N$,
the coordinate representation of the differential is the Jacobian matrix $J_{F,ij} = \partial F^i/\partial x^j$.
The component matrix of the pullback metric is then given by:
\begin{equation} \label{eq:matrix_pullback}
  [F^*\ncal] = J_F^T[\ncal]J_F.
\end{equation}
If $F$ is a diffeomorphism (i.e., if it also has a smooth inverse) and satisfies $\mcal = F^*\ncal$,
then $F$ is called a \emph{(Riemannian) isometry} between $(M,\mcal)$ and $(N,\ncal)$.
In other words, the isometry $F$ preserves the metric from $M$ to $N$.
As a result, angles, lengths, and volumes on isometric manifolds $(M,\mcal)$ and $(N,\ncal)$ are measured in the same way.

\section{Current definitions of unit simplices}
\label{sec:prop_unit_simplices}
Unit simplices play a central role in the continuous mesh framework.
Indeed, unit simplices with respect to an error-minimizing metric
are endowed with the same error and yield error-minimizing triangulations.
These triangulations are the discrete counterparts of the error-minimizing metric and form an equivalence class,
whose equivalence relation is to consist of unit simplices with respect to the metric.
In this section, we summarize the current definitions of both linear and curved unit simplices,
and review their properties and limitations.

\subsection{Linear simplices} As mentioned in the previous section,
linear, or straight-sided, simplices are said to be \emph{unit}
if their edges have unit length with respect to an input metric.
Formally:
\begin{definition} (Linear unit simplex) \label{def:linear_unit}
  Let $(\Omega,\mcal)$ with $\Omega \subset \Rn$ be a Riemannian $n$-manifold with constant metric.
  A linear $n$-simplex $K \subseteq \Omega$ with edges $e_i$, $1 \leq i \leq N_e \triangleq n(n+1)/2$, is \emph{unit} with respect to $\mcal$ if $\ell_\mcal(e_i) = 1$,
  that is, if all its edges have unit length with respect to $\mcal$.
\end{definition}

The set of linear unit simplices for a constant metric $\mcal$ is explicitly given by \cite{loseille2008adaptation}:
\begin{equation}\label{eq:set_linear_unit_simplices}
  \kcal_\mcal \triangleq \left\{ K = \mcalmud R \keq + \bfx ~ | ~ R \in \text{SO}(n), ~\bfx \in \Rn \right\}.
\end{equation}
That is,
they are, up to a translation, the image of $\keq$ under an arbitrary rotation $R$ followed by the linear transformation $\mcalmud \triangleq P\Lambda^{-1/2}P^T,$
where $[\mcal] = P\Lambda P^T$ is the diagonalized form of the component matrix of $\mcal$.
Linear unit simplices share the following properties \cite{loseille2008adaptation, rocheryThesis}:

\begin{proposition}\label{prop:prop_linear_unit}
Let $K$ be a linear $n$-simplex unit for the constant metric $\mcal$.
\begin{enumerate}
  \item If $n > 1$, the inner angles of $K$ are each $\pi/3$;
  \item $|K|_\mcal = |\keq|_{\eucl}$; 
  \item Let $H \in \Run^{n \times n}$ be a symmetric matrix. Then:
  \begin{equation}
    \sum_{i=1}^{N_e} e_i^T H e_i = \frac{n+1}{2} \trace \left(\mcalmud H \mcalmud\right).
  \end{equation}
  \item There is a rotation matrix $R$ such that the constant Jacobian matrix of their reference-to-physical transformation $J_K = J_{\khat \to K}$ is:
  \begin{equation}
    J_K = \mcalmud R \jeq.
    \label{eq:jacobian_unit_linear}
  \end{equation}
\end{enumerate}
\end{proposition}
\begin{proof}
  See \cite{loseille2008adaptation, loseille2011continuous, rocheryThesis}.
\end{proof}
The first two invariants simply state that unit simplices have the same inner angles and area.
When $H$ is positive-definite, the third invariant translates the fact that the energy of the edges with respect to the metric tensor $\hcal$ with component matrix $H$, is an invariant.
This is used in e.g. \cite{loseille2008adaptation, loseille2011continuous, coulaud2016very, rochery2024generalized, bawin2024metric}
to derive interpolation error estimates on unit simplices, with $H$ representing the Hessian or higher-order derivatives of the field of interest.
The fourth invariant follows directly from the characterization \eqref{eq:set_linear_unit_simplices},
but is far-reaching nonetheless and can be readily generalized to nonlinear simplices and nonconstant metrics \cite{rocheryThesis, rochery2024generalized}, see the next section.
Since $R$ and $\jeq$ are constant matrices, $J_K$ is constant if and only if the metric is constant.
Thus, constant metrics always yield linear unit simplices,
and, conversely, linear simplices are unit only for a constant metric.

\subsection{Curved simplices} \label{sec:curvedSimplices}
When considering curved (e.g., $\pdeggeo$) simplices,
Definition \ref{def:linear_unit} becomes insufficient, as it only prescribes the global edge length.
For linear simplices, two vertices are connected by a straight line with a linear parameterization,
and a definition based solely on the edges ensures that Proposition \ref{prop:prop_linear_unit} holds for all unit simplices.
In contrast, the vertices of a curved simplex may be connected by arbitrary smooth curves.
An edge-based definition alone is inadequate,
as any nontrivially curved triangle with unit edges under a constant metric has a nonconstant Jacobian matrix, thus violating \eqref{eq:jacobian_unit_linear}.
To generalize unitness to curved simplices,
it was recently proposed to adopt \eqref{eq:jacobian_unit_linear} as a definition rather than as an invariant \cite{rocheryThesis, rochery2024generalized}:
\begin{definition} (Curved unit simplex) \label{def:unit_curved_triangle_jacobian}
  Let $(\Omega,\mcal)$ with $\Omega \subset \Rn$ be a Riemannian $n$-manifold.
  A curved $n$-simplex $K \subseteq \Omega$ is unit with respect to $\mcal$ if there is a rotation matrix $R$ such that for any $\bfxi \in \khat$:
  \begin{equation}
    J_K(\bfxi) = \mcalmud(F_K(\bfxi)) R \jeq,
    \label{eq:unit_curved_triangle_jacobian}
  \end{equation}
  and if $\mcal$ is such that $F_K$ is a smooth function of $\bfxi$.
\end{definition}
This defines a unit simplex in terms of a pointwise relation between the metric and the Jacobian matrix $J_K$.
Since $F_K$ is smooth,
the metric must exhibit the required symmetries so that $F_K$ has commuting partial derivatives of arbitrary order.
These symmetries are discussed in \cite{rochery2024generalized} for the special case of quadratic ($\pdeux$) simplices.
Because $J_K$ appears twice in the computation of edge lengths in \eqref{eq:length_in_metric},
one checks that simplices described by Definition \ref{def:unit_curved_triangle_jacobian} have edges of unit length with respect to $\mcal$,
and are thus also unit in the sense of Definition \ref{def:linear_unit},
see \cite{bawin2024metric}.

Definition \ref{def:linear_unit} is appealing for practical mesh generation,
as it defines unitness through a local edge-based criterion.
This enables fast algorithms that rely solely on edge splits and collapses
until all mesh edges have lengths within $[1/\sqrt{2},\sqrt{2}]$.
In contrast, Definition \ref{def:unit_curved_triangle_jacobian} imposes a pointwise constraint on curved simplices,
without explicitly specifying the edge parameterization.

In \cite{aparicio2019defining, aparicio2022high}, curved triangulations in two and three dimensions are obtained from a linear mesh
by curving the edges to minimize an elementwise measure of distortion with respect to $\keq$.
Simplices of minimal distortion satisfy $J_K = c\mcalmud R J_\triangle$ with $c > 0$,
meaning they are conformal transformations of $\keq$ and have controlled inner angles.
The appropriate scaling of $J_K$ is achieved by controlling the edge lengths with respect to $\mcal$ on the underlying linear mesh,
that is, by curving a linear anisotropic mesh adapted to $\mcal$ in the usual sense.
The edge curving step, however, requires minimizing a global mesh distortion,
which is a coupled and expensive problem.
In this regard, an edge-based criterion for unitness would be particularly beneficial to make high-order mesh adaptation tractable for large-scale applications.

Alternatively,
a promising local edge curving scheme was proposed in \cite{rochery2024generalized},
by decoupling the linear and high-order parts of the mesh elements.

\section{Isometric simplices and meshes}
\label{sec:isometric_simplices}
We now propose a general definition of unitness as the image of an ideal simplex by a Riemannian isometry.
This definition includes Definitions \ref{def:linear_unit} and \ref{def:unit_curved_triangle_jacobian} as special cases,
with their properties recovered from the local length-preserving nature of isometries.
Simplices that are isometric to a linear simplex in the reference space have geodesic edges,
yielding an edge-based criterion for the generation of curved unit simplices.
Unit meshes, the targets of metric-based adaptation,
are defined as the image of a uniform tiling of $\Rn$ under an isometry.
Isometric simplices and meshes exist only if there is an isometry between the Euclidean space
and the physical space equipped with the Riemannian metric of interest.
This is possible only if the metric has zero curvature, as curvature is an invariant of isometric manifolds.
We discuss the existence of isometric meshes on such manifolds,
then introduce two relaxed definitions for quasi-unitness.
These definitions can be used to generate quasi-isometric meshes for practical applications with metrics with nonzero curvature.

\subsection{Isometric unit simplices}

Linear unit simplices preserve the edge lengths of the ideal simplex $K_\triangle$,
while \cref{def:unit_curved_triangle_jacobian} describes a curved unit simplex through its Jacobian matrix $J_K = \mcalmud R J_\triangle = J_{\triangle \to K}J_\triangle$.
Because $J_{\triangle \to K}$ satisfies \eqref{eq:matrix_pullback} for $[F^*\mcal] = I = [\eucl]$,
that is, because it is associated to a transformation that pulls the metric $\mcal$ to the Euclidean metric,
we are naturally led to the following definition of a unit simplex:

\begin{definition} \label{def:unit_isometry}
  Let $(\widehat{\Omega},\eucl)$, $\widehat{\Omega} \subset \Rn$, and $(M,\mcal)$ be Riemannian $n$-manifolds, and let $K_0 \subset \widehat{\Omega}$ be a positively oriented linear $n$-simplex.
  A curved $n$-simplex $K \subseteq M$ is unit with respect to both $\mcal$ and $K_0$ if it is isometric to $K_0$,
  that is, if it is the image $K = F(K_0)$
  for some 
  orientation-preserving isometry $F:\widehat{\Omega} \to M$.
\end{definition}

A unit simplex $K$ with respect to $\mcal$ and some simplex $K_0$ in the reference space is thus simply the image of $K_0$ by some isometry $F$,
provided such an isometry between $(\widehat{\Omega},\eucl)$ and $(M,\mcal)$ exists, which is discussed hereafter.
Note that $F$ should preserve the orientation, so as not to invert the simplex $K$, making it invalid (in the sense of high-order finite elements).
Before discussing this definition, we give the main properties of unit simplices defined by an isometry:

\begin{proposition} \label{prop:properties_isometric_elements}
Let $K = F(K_0)$ be unit with respect to $\mcal$ and $K_0$. 
\begin{enumerate}
  \item $K$ is positively oriented. In particular, if $K$ is a $\pgeo$ simplex, it is valid everywhere.
  \item The inner angles and area of $K_0$, as well as the lengths and energies of its edges, are preserved in $K$.
  \item $K$ is a geodesic simplex in $(M,\mcal)$.
  \item There is a rotation matrix $R$ such that the Jacobian matrix of $F_K$ writes for any $\bfxi \in \khat$:
  \begin{equation}
    J_K(\bfxi) = \mcalmud(F_K(\bfxi)) R J_{\khat \to K_0}.
  \end{equation}
  \item Let $\eta_{K_0}(\bfy)$ be the pointwise distortion of $K$ with respect to $K_0$ defined for $\bfy \in K_0$ by \cite{aparicio2019defining}:
  \begin{equation}
    \eta_{K_0}(\bfy) \triangleq \frac{1}{n} \frac{\trace \,\left(J_{K_0 \to K}^T(\bfy) \,[\mcal_{F(\bfy)}]\, J_{K_0 \to K}(\bfy)\right) }{\left[\det \left(J_{K_0 \to K}^T(\bfy) \,[\mcal_{F(\bfy)}]\, J_{K_0 \to K}(\bfy)\right)\right]^{1/n}} \geq 1.
    \label{eq:distortionK0}
  \end{equation}
  Then $\eta_{K_0}(\bfy) = 1$ for all $\bfy \in K_0$.
\end{enumerate}
\end{proposition}
\begin{proof}
    \hfill
  \begin{enumerate}
    \item $K_0$ is positively oriented and $F$ preserves the orientation.
    A $\pgeo$ simplex is valid if it is positively oriented everywhere, thus $K$ is valid if
    it happens to be a $\pgeo$ simplex.
    \item This is a direct consequence of each quantity being preserved by Riemannian isometries,
    which are in particular metric isometries. 
    \item The edges of $K_0$ are straight lines, which are geodesics in
    $\widehat{\Omega}$ with the Euclidean metric.
    Since isometries map geodesic curves to geodesic curves \cite{lee2012smooth},
    $K = F(K_0)$
    has geodesic edges in $(M,\mcal)$.
    \item This is a consequence of the metric-preserving property of isometries.
    As an isometry between $(\widehat{\Omega},\eucl)$ and $(M,\mcal)$, $F$ pulls the metric $\mcal$ back to the Euclidean metric on $\widehat{\Omega}$,
    that is, $F^*\mcal = \eucl$.
    Using \eqref{eq:matrix_pullback} and the fact that the component matrix of $\eucl$ is the identity matrix $I$, we have:
    \begin{equation}
      [F^*\mcal] = J_F^T [\mcal] J_F = [\eucl] = I.
      \label{eq:pullback_to_euclidean}
    \end{equation}
    This is satisfied when the Jacobian matrix is of the form $J_F = \mcalmud O = P\Lambda^{-1/2}P^TO$, with $O$ an orthogonal matrix satisfying $O^TO = OO^T = I$.
    Indeed, since $P$ is also orthogonal:
    \begin{equation}
      J_F^T [\mcal] J_F = O^TP\Lambda^{-1/2}P^TP\Lambda P^TP\Lambda^{-1/2}P^TO = I,
    \end{equation}
    thus $J_F =  \mcalmud O =  (\mcalmud \circ F)O$.
    Orthogonal matrices include rotations, reflections and their compositions, however, reflections have negative determinant, so that $F$ is no longer orientation-preserving.
    Thus, we only consider the orthogonal matrices $R \in SO_n$ associated to rotations, so that the Jacobian matrices of interest are of the form:
    \begin{equation}
      J_F = \mcalmud R.
    \end{equation}
    As $F_K = F_{\khat \to K} = F \circ F_{\khat \to K_0}$,
    one has $J_K = J_FJ_{\khat \to K_0} = \mcalmud R J_{\khat \to K_0}$.
    \item The distortion $\eta_{K_0}(\bfy)$ is the AM-GM inequality for the (strictly positive) eigenvalues of $[F^*\mcal]$:
    \begin{equation}
      \frac{1}{n}\trace\,[F^*\mcal] \geq \left( \det [F^*\mcal] \right)^{1/n},
    \end{equation}
    with equality $\eta_{K_0}(\bfy) = 1$ only if $[F^*\mcal] = cI$ for some $c > 0$.
    From \eqref{eq:pullback_to_euclidean}, we have $[F^*\mcal] = I$, thus the distortion is minimal and unit on $K_0$.
  \end{enumerate}
\end{proof}

The previous definitions of unitness are recovered by setting the ideal simplex to $K_0 = \keq$.
Indeed, $\keq$ has unit edges in the reference space $(\widehat{\Omega},\eucl)$ and $F$ preserves the edge lengths,
thus recovering \cref{def:linear_unit},
and the fourth property of \cref{prop:properties_isometric_elements} is precisely \cref{def:unit_curved_triangle_jacobian}.
Similarly, the invariants shared by linear unit triangles, \cref{prop:prop_linear_unit},
are direct consequences of Riemannian isometries preserving the inner products and lengths (and thus the angles) of vectors,
the area of compact manifolds, and the energy of curves \cite{bawin2024metric}.

Riemannian isometries impose a strong constraint on simplices,
as distance preservation is enforced locally,
unlike the linear definition where only the global edge length is preserved between $K_0$ and $K$.
As a result, an isometry $F$ between the Euclidean space $(\widehat{\Omega},\eucl)$
and a given manifold $(M,\mcal)$ does not generally exist,
and when it does, it typically cannot be expressed in closed form.
Manifolds $M \subseteq \Rn$ with a constant metric are a notable exception:
in this case, $F$ is linear and given explicitly by $F(\bfxi) = \mcalmud R\bfxi$.
Unit simplices are then the image $F(K_0) = \mcalmud R K_0 + \bfx$, as described by \eqref{eq:set_linear_unit_simplices}.
Similarly,
isometric simplices are generally not polynomial ($\pdeggeo$) simplices,
unless the image manifold allows it, see Section \ref{sec:examples}.

As the images of a linear ideal simplex $K_0$ under isometries,
unit simplices are geodesic simplices, that is, their edges are geodesic curves with minimal length with respect to $\mcal$.
This property has been exploited in recent work on high-order meshing to generate curved meshes through length minimization
\cite{zhang2018curvilinear, zhang2021generation, rocheryThesis, bawin2023optimally},
but was never formally established.
Since $K_0$ is linear and thus geodesic for the Euclidean metric,
the parameterizations of the edges of $K$ are restricted to geodesic curves, preserving the local distance-minimizing property of the edges of $K_0$.

In the same way that a single metric tensor describes an equivalence class of linear unit elements \cite{loseille2008adaptation},
a Riemannian metric describes an equivalence class of curved unit elements, where the equivalence relation is to be isometric
to one another.
Indeed, if $F_1, F_2$ are distinct isometries with Jacobian matrices $J_F = \mcalmud R_1$ and $\mcalmud R_2$,
their inverse are also isometries, so that the unit elements $K_1 = F_1(K_0)$ and $K_2 = F_2(K_0)$ are isometric and related by $K_2 = F_2(F_1^{-1}(K_1))$.

\subsection{Unit meshes}
In the continuous mesh framework \cite{loseille2008adaptation, loseille2011continuous, loseille2011continuous2},
a unit mesh consists exclusively of unit linear simplices and is \emph{regular}, meaning that it contains only edges of unit length.
Such meshes exist in one and two dimensions but not in three dimensions,
as there is no tessellation of $\Rtrois$ composed only of regular tetrahedra.
For practical applications,
nonregular tetrahedra that pave the space are considered \cite{loseille2008adaptation}.
In this paper, since Definition \ref{def:unit_isometry} defines unitness based on an arbitrary ideal simplex $K_0$,
a unit mesh is understood as a tiling composed only of copies of images of $K_0$:
\begin{definition}
  A unit (simplicial) mesh on an $n$-manifold is the image of a uniform simplicial tiling of $\Rn$ by an orientation-preserving isometry.
\end{definition}
Such unit meshes are regular only when $K_0$ is a regular simplex.
\Cref{fig:ref_tilings} shows two examples of ideal simplicial tilings of $\R$:
a regular tiling $\trieq$ made of copies of the equilateral triangle $\keq$,
and a nonregular tiling $\trirec$ made of copies of the reference right triangle $\khat$.
Each tiling generates an equivalence class of unit meshes, which are obtained as their image by an isometry.

\begin{figure}[tbhp]
  \centering
  \includegraphics[width=0.8\linewidth]{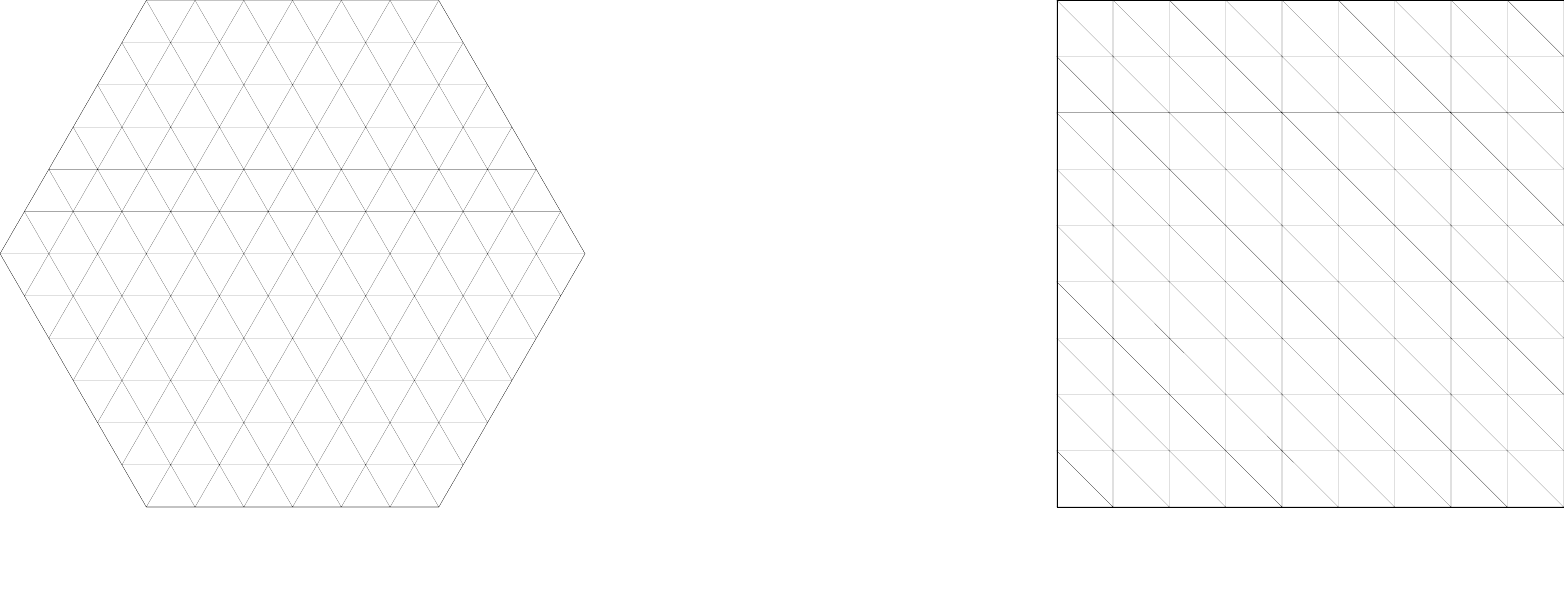}
  \caption{Uniform tilings $\trieq$ (left) and $\trirec$ (right) of $\R$.}
  \label{fig:ref_tilings}
\end{figure}


\subsection{Existence of isometric simplices and unit meshes}
\label{sec:existence}
By construction,
unit simplices and unit meshes exist on a manifold $M$ only if it is isometric to $\Rn$.
Such manifolds and their metric are said to be \emph{flat}, and are characterized by an absence of curvature,
in the sense that their \emph{Riemann curvature 4-tensor} vanishes everywhere \cite{lee2018introduction}.
This tensor has 0, 1 and 6 independent components in dimension $n=1,2,3$ respectively.
Thus, all 1-manifolds are flat and unit simplices exist for all Riemannian metrics,
leaving aside the considerations of the boundaries.

The curvature on 2-manifolds is completely determined by their scalar curvature, which is given (up to a factor 2)
by their Gaussian curvature.
Surfaces with zero Gaussian curvature, such as cones and cylinders, are called \emph{developable},
as they can be obtained by folding the plane without tearing it.
Perfectly unit meshes can be constructed on these surfaces, as illustrated in Section \ref{sec:examples}.
If $F$ is an isometry from $(\widehat{\Omega},\eucl)$ to such a 2-manifold and $R_0$ is any rotation matrix,
a unit mesh is obtained from $\trieq$ by first rotating it by $R_0$,
then applying $F$. Thus, $F(R_0\trieq)$ is a unit mesh, and as the image of an equilateral tiling, its elements have unit edge length in the deformed space.
Since two triangles in $\trieq$ are rotations of 0 or $\pi/3$ of one another, the Jacobian matrices $J_K$ of the triangles in $F(R_0\trieq)$
are explicitly given by $J_K = \mcalmud R_0\jeq$ and $J_K = \mcalmud R_0R(\frac{\pi}{3})\jeq$.

In two dimensions, Properties 2. and 3. of Proposition \ref{prop:properties_isometric_elements}
show that a curved triangle $K$ that is unit with respect to $K_0 = \keq$ and a metric $\mcal$
is a geodesic triangle with unit edges.
If the converse was true,
then one could generate curved unit simplices only by constructing geodesic edges of the right length.
This only holds in the narrow case of flat manifolds, however,
for which we have the equivalence result:

\begin{proposition}
    \label{prop:equivalence_unite_geodesique}
    Let $(M,\mcal)$ be a flat 2-manifold and let $K_0 \subset \widehat{\Omega}$ be a linear triangle.
    A curved triangle $K \subseteq M$ is isometric to $K_0$ iff it is a geodesic triangle whose homologous edges have the same length as those of $K_0$.
\end{proposition}
\begin{proof}
The $\Rightarrow$ direction is simply applying the properties of an isometry, which preserves the lengths and geodesics.
To prove the $\Leftarrow$ direction, we proceed as follows.
Since $M$ is flat, there is an isometry $F : \R \to M$ and a triangle $K_0'$ in $\R$ such that $K = F(K_0')$.
The lengths of the edges of $K_0'$ are the same as those of $K$, and thus of $K_0$.
By the side-side-side theorem for congruent triangles in the Euclidean space,
there is a Euclidean plane isometry $G$ such that $K_0' = G(K_0)$, so that $K = F(G(K_0))$ is isometric to $K_0$ by composing both isometries.
\end{proof}

In particular, for a constant metric $\mcal$,
linear triangles are geodesic, thus linear triangles with unit edges are unit with respect to $\mcal$ and $K_\triangle$,
which agrees with the usual definition of the continuous mesh framework.
This does not hold on nonflat manifolds, as any compact region with nonzero Gaussian curvature
enclosed in a geodesic simplex $K$ prevents the existence of an isometry with the Euclidean space.
Thus, for arbitrary metrics, controlling only the edges length and geodesicity is not sufficient to generate isometric simplices.
To create (quasi-)isometric simplices,
we propose in the following to control either a quality measure, similarly to \cite{aparicio2019defining, aparicio2022high}, or the inner angles in addition to the edges length and geodesicity.

Lastly, 3-manifolds, which are not detailed in this paper, are flat only if their \emph{Ricci curvature tensor} vanishes.
This tensor is a symmetric 2-tensor with 6 independent components, obtained by contracting two indices of the Riemann curvature tensor \cite{lee2018introduction}.

\subsection{Quasi-unit simplices}
\label{sec:quasiunit}

Definition \ref{def:unit_isometry} is elegant and provides a broader framework for unit simplices, but is overly restrictive,
as flat metrics represent an insignificant fraction of the Riemannian metrics at play when performing mesh adaptation for real-life applications.
For the practical generation of unit curvilinear meshes, a generalization of quasi-unitness is needed.
As unit elements should be isometric to an ideal simplex,
a certain number of relaxed properties from Proposition \ref{prop:properties_isometric_elements}
should be satisfied by quasi-unit elements.

In the linear case,
quasi-unit elements preserve the relaxed global edge lengths with respect to $\keq$,
meaning their edge lengths lie within $[1/\sqrt{2},\sqrt{2}]$.
In practice, the inner angles are also globally controlled to avoid the creation of sliver elements with volumes approaching zero.
However, individual inner angles or the geodesicity of the edges, which are local properties,
are not controlled.

In addition to edge lengths, quasi-unit curvilinear elements should also preserve relaxed local properties,
such as inner angles, distortion or quasi-geodesic edges.
In the following, we propose two definitions of quasi-unit curvilinear elements:
a stronger definition in terms of a distortion-based quality,
and a weaker one in terms of geodesic edges with controlled inner angles.

In \cite{loseille2008adaptation, loseille2011continuous}, the sliverness of linear simplices is controlled by the following quality function, tailored for an equilateral ideal simplex in dimension $n=3$:
\begin{equation} \label{eq:quality_loseille}
  Q_{\keq}(K) = \frac{36}{\sqrt[3]{3}} \frac{|K|_\mcal^{2/3}}{\sum_{i=1}^6 \ell_\mcal(e_i)^2} \in [0,1].
\end{equation}
This quality measure can be generalized to higher-order simplices by an elementwise distortion-based quality with respect to the generic ideal simplex $K_0$.
In dimension $n$, this quality is defined as follows from the pointwise distortion \eqref{eq:distortionK0}, similarly to \cite{aparicio2019defining, aparicio2022high}:
\begin{equation}
 Q_{K_0}(K) \triangleq \left( \frac{1}{|K_0|_{\eucl}} \int_{K_0} \eta_{K_0}^n(\bfy)\,d\bfy \right)^{-1/n} \in [0,1].
 \label{eq:elementwise_distortion_based_quality}
\end{equation}
For a constant metric and $K_0 = \keq$,
this reduces to \eqref{eq:quality_loseille}.
Indeed, we have:
\begin{equation}
  \det (\jeqK^T \,[\mcal]\, \jeqK)=(\det \jeqK)^2 \det \mcal,
\end{equation}
and, similarly to the third invariant of Proposition \ref{prop:prop_linear_unit} for linear simplices\footnote{
In particular, the Jacobian matrix is $\jeqK = \mcalmud R$ for linear unit simplices, thus
$\trace \, (\jeqK^T \,[\mcal]\, \jeqK) = \trace \, (R^T \mcalmud\,[\mcal]\, \mcalmud R) =
\trace \, (\mcalmud \,[\mcal]\, \mcalmud) = n$ and we recover the result of Proposition \ref{prop:prop_linear_unit} with $H = [\mcal]$.
}:
\begin{equation}
    \sum_{i = 1}^{N_e} \ell_\mcal(e_i)^2 = \sum_{i = 1}^{N_e} e_i^T[\mcal]e_i
    = \frac{n+1}{2}\,\trace \, \left(\jeqK^T \,[\mcal]\, \jeqK\right).
\end{equation}
For linear simplices, the distortion is constant on $K_0$ and we have, letting $J = \jeqK$:
\begin{equation}
  \begin{aligned}
    \int_{K_0} \eta_{K_0}^n\,d\bfy
    =
    \eta_{K_0}^n |K_0|_{\eucl}
    &=
    \frac{\trace^n(J^T[\mcal]J)}{n^n\,(\det J)^2 \det \mcal} \frac{|K_0|_{\eucl}^3}{|K_0|_{\eucl}^2}\\
    &=
    \frac{\trace^n(J^T[\mcal]J)}{n^n|K|^2_\mcal} |K_0|_{\eucl}^3,
  \end{aligned}
\end{equation}
thus:
\begin{equation}
  Q_{K_0} = n|K_0|_{\eucl}^{-2/n} \frac{|K|^{2/n}_\mcal}{\trace(J^T[\mcal]J)} = N_e|K_0|_{\eucl}^{-2/n} \frac{|K|^{2/n}_\mcal}{\sum_{i = 1}^{N_e} \ell_\mcal(e_i)^2} \in [0,1].
\end{equation}
For $K_0 = \keq$, we recover \eqref{eq:quality_loseille} with the normalization coefficients:
\begin{equation}
  N_e|\keq|_{\eucl}^{-2/n} = \frac{12}{\sqrt{3}} ~~~\text{for}~~n=2 ~~~\text{and}~~~\frac{36}{\sqrt[3]{3}}~~~\text{for}~~n=3.
\end{equation}
Quasi-unit simplices can then be extended as follows:

\begin{definition}
  (QU1)
  A curved simplex $K$ is quasi-unit with respect to $\mcal$ and $K_0$ if
  \begin{itemize}
    \item its quality $Q_{K_0}$ lies within $[a,1]$ for $0 < a < 1$;
    \item each edge $e$ of $K$ homologous to the edge $e_0$ of $K_0$ satisfies:
    \begin{equation}
      \frac{1}{\sqrt{2}} \ell_{\eucl}(e_0) \leq \ell_\mcal(e) \leq \sqrt{2}\ell_{\eucl}(e_0).
      \label{eq:quasi_isometry}
    \end{equation}
  \end{itemize}
\end{definition}
This definition controls both the distortion, thus the inner angles, and the edge lengths of the simplex $K$, ensuring that it is quasi-isometric to $K_0$.
In \cite{loseille2008adaptation, loseille2011continuous}, the constant $a$ is set to 0.8 to allow for a variety of families of tetrahedra to be considered for three-dimensional meshing.

Alternatively, quasi-unitness can be addressed through the geodesic property of the edges.
From the Gauss-Bonnet formula \cite{lee2018introduction},
geodesic triangles preserve the sum of their interior angles only when the enclosed Gaussian curvature of the metric cancels out,
which is not expected in general.
Instead, we propose to preserve the relaxed angles individually as follows:
\begin{definition}
  (QU2)
  A curved simplex $K$ is quasi-unit with respect to $\mcal$ and $K_0$ if
  \begin{itemize}
    \item it is a geodesic simplex;
    \item each angle $\theta$ homologous to the angle $\theta_0$ in $K_0$ satisfies
    $\theta_0 - b \leq \theta \leq \theta_0 + b;$
    \item its homologous edges with respect to $K_0$ satisfy \eqref{eq:quasi_isometry}.
  \end{itemize}
\end{definition}
In this definition, $b$ is an angle increment allowing the practical generation of such quasi-unit simplices.
In the case $K_0 = \keq$, for instance, our results \cite{bawin2024metric}
indicate that quasi-unit meshes can be generated with a vast majority of angles lying within $60^\circ \pm 30^\circ$,
Providing a more accurate value for the increment $b$ will be the object of future work.

Clearly, a unit simplex in the sense of \cref{def:unit_isometry} is quasi-unit in the sense of both \emph{QU1} and \emph{QU2}.
Imposing the inner angles is a weaker condition than controlling the distortion,
and geodesicity is a local property,
thus \textit{QU2} is weaker than \textit{QU1} in terms of pointwise conformity to $K_0$.
In practice, however, \textit{QU1} requires integration in a varying metric field and is thus costlier to enforce.
Additionally, the metric is often \emph{graded} \cite{alauzet2010size}, limiting its variations and curvature.
As a result, \textit{QU2} is expected to be fast
to generate while remaining a satisfying indicator of unitness.

In standard anisotropic mesh adaptation, all straight simplices are also $\pun$ simplices,
that is, all straight lines are represented by a linear parameterization,
and there is no distinction between a simplex and a polynomial simplex.
For curvilinear simplices, however, geodesic edges do not usually have a polynomial parameterization,
and there should be a distinction between quasi-unit simplices and quasi-unit $\pgeo$ simplices.
The definitions above can be adapted for $\pgeo$ simplices as follows:

\begin{definition}
  (\textit{QU1p})
  A $\pgeo$ simplex $K$ is quasi-unit w.r.t. $\mcal$ and $K_0$ if
  \begin{itemize}
    \item its quality $Q_{K_0}$ is in $[a,1]$;
    \item if homologous edges $e$ and $e_0$ satisfy \eqref{eq:quasi_isometry}.
  \end{itemize}
\end{definition}

\begin{definition}
  (\textit{QU2p})
  A $\pgeo$ simplex $K$ is quasi-unit w.r.t. $\mcal$ and $K_0$ if:
  \begin{itemize}
    \item it is valid;
    \item its edges approximate the geodesics, thus are of minimal length;
    \item homologous inner angles $\theta$ and $\theta_0$ satisfy $\theta_0 - b \leq \theta \leq \theta_0 + b;$
    \item homologous edges $e$ and $e_0$ satisfy \eqref{eq:quasi_isometry}.
  \end{itemize}
\end{definition}

\textit{QU1p} is the immediate polynomial version of \textit{QU1},
whereas \textit{QU2p} approximates a geodesic simplex with polynomial edges.
Note that enforcing validity is only required for \textit{QU2p} high-order simplices,
as \textit{QU1p} simplices have a strictly positive quality, which guarantees their validity.

\section{Examples of unit isometric triangulations}
\label{sec:examples}

We conclude by illustrating our new definition of unit element and unit meshes in two dimensions.
As a proof of concept, unit triangulations that are isometric to a uniform tiling of $\R$ with copies of either $\keq$ or $\khat$ are obtained through optimization.
To this end, each element of an initial triangulation is divided into a uniform subtriangulation,
and a Riemannian isometry is approximated by locally controlling lengths and areas.
A cost function measures the discrepancy between the ideal and current metric-weighted edge lengths of a subtriangle, while also accounting for the area of the subtriangle to ensure validity throughout the optimization.

The metric is induced by a variety of developable and non-developable surfaces embedded in $\Rtrois$.
The optimized meshes are curved in the parameter space in $\R$,
and their image on the surfaces consist of (quasi-)geodesic curves.

\subsection{Cost function and minimization}
Isometric meshes preserve distances and areas locally with respect to their preimage triangulation.
To mimic this property,
the local mesh size is controlled by dividing each triangle $K$ of an initial uniform triangulation into a subgrid of $N^2$ subtriangles.
Each subtriangle $k$ of $K$ is also isometric to a subtriangle of $K_0$,
whose edge lengths and area decrease with $1/N$ and $1/N^2$, respectively.
Depending on the ideal simplex $K_0$,
the target metric-weighted edge lengths and area of a subtriangle are thus:
\begin{equation}
  \ell_{\text{target}} = \frac{1}{N} ~~~\text{and} ~~~|k|_\text{target} = \frac{|\keq|_{\eucl}}{N^2} = \frac{\sqrt{3}}{4N^2}~~~~\text{if}~~K_0 = \keq,
\end{equation}
and
\begin{equation}
  \ell_{\text{target}} = \frac{1}{N} ~~\text{or}~~\frac{\sqrt{2}}{N} ~~~\text{and} ~~~|k|_\text{target} = \frac{|\krec|_{\eucl}}{N^2} = \frac{1}{2N^2}~~~~\text{if}~~K_0 = \krec.
\end{equation}
Note that $\krec$ requires differentiating between the hypotenuse and the edges adjacent to the right angle.

In addition to controlling the edge lengths,
control over the subtriangles area is required to avoid scrambling the elements during the optimization.
This is done by devising a cost function $E = \ecal + \fcal$ composed of two terms, similarly to \cite{toulorge2013robust}:
an energy $\ecal$ associated with the edge lengths,
and a penalization term $\fcal$ to enforce validity of the subtriangles.

The energy $\ecal$ measures the discrepancy between $\ell_{\text{target}}$ and the metric-weighted length of the edges of each subtriangle.
If the chosen reference tiling is $\trieq$, the target length is the same for all subedges and the energy of the triangulation $\tcal$ writes:
\begin{equation}
  \begin{aligned}
    \ecal_{\triangle}(\tcal) &= \frac{1}{2} \sum_{\text{edges}~\bfe} \left( \ell_\mcal^2(\bfe) - \frac{1}{N^2} \right)^2
  \end{aligned}
\end{equation}
For $\trirec$, the target length depends on whether the edge is mapped from one of the edges adjacent to the right angle $\bfe_\text{right}$ or from the hypotenuse $\bfe_\text{hyp}$:
\begin{equation}
  \begin{aligned}
    \ecal_{\khat}(\tcal) &= \frac{1}{2} \left(\sum_{\bfe_\text{right}} \left( \ell_\mcal^2(\bfe_\text{right}) - \frac{1}{N^2} \right)^2
                                                         + \sum_{\bfe_\text{hyp}} \left( \ell_\mcal^2(\bfe_\text{hyp}) - \frac{2}{N^2} \right)^2 \right).
  \end{aligned}
\end{equation}
The edge length is approximated by evaluating the metric at the middle of the edge, which is satisfactory if the subtriangles are fine enough.

The penalization term $\fcal$ controls the area of the subtriangles.
The optimal metric-weighted area of these subtriangles is $|k|_\text{target}$,
and they should remain valid throughout the optimization.
Indeed, minimizing only $\ecal$ leads to subtriangles with isometric edges
but that may be inverted or superimposed.
Validity of the linear subtriangles is ensured by enforcing that their area remains strictly positive.
As in \cite{toulorge2013robust}, we use a logarithmic barrier \cite{freitag2002comparison} to penalize very small subtriangles
and a quadratic function to penalize the very large ones.
This yields the penalization term:
\begin{equation}
  \fcal(\tcal) = \sum_{K \in \tri}\sum_{k \in K} F_\epsilon\left( \frac{|k|_\mcal}{|k|_\text{target}} \right)
  \quad\text{with}\quad
  F_\epsilon(x) = \left[\ln\left( \frac{x - \epsilon}{1 - \epsilon}\right)\right]^2 + \left( x - 1 \right)^2.
\end{equation}
The behavior of $F_\epsilon$ is shown in \Cref{fig:logbarrier}: it vanishes when $|k|_\mcal = |k|_\text{target}$,
blows up when $|k|_\mcal \to \epsilon$ and increases as $|k|_\mcal \to \infty$.
The parameter $\epsilon < 1$ sets the position of the barrier and allows for dynamic control of $\min |k|_\mcal$ to speed up the optimization.
The first log-barrier is set with $\epsilon = 0$ to guarantee that all subtriangles remain valid.
This is an adequate first value, as the starting linear mesh is valid.
After the first optimization is completed, $\epsilon$ is increased to $\min(0.99, 0.99 \times |k|_\mcal/|k|_\text{target})$, pushing the minimum area towards 1 during the next optimization.
The log-barrier is moved this way 5 times.
\begin{figure}[tbhp]
  \centering
  \includegraphics[width=0.4\linewidth]{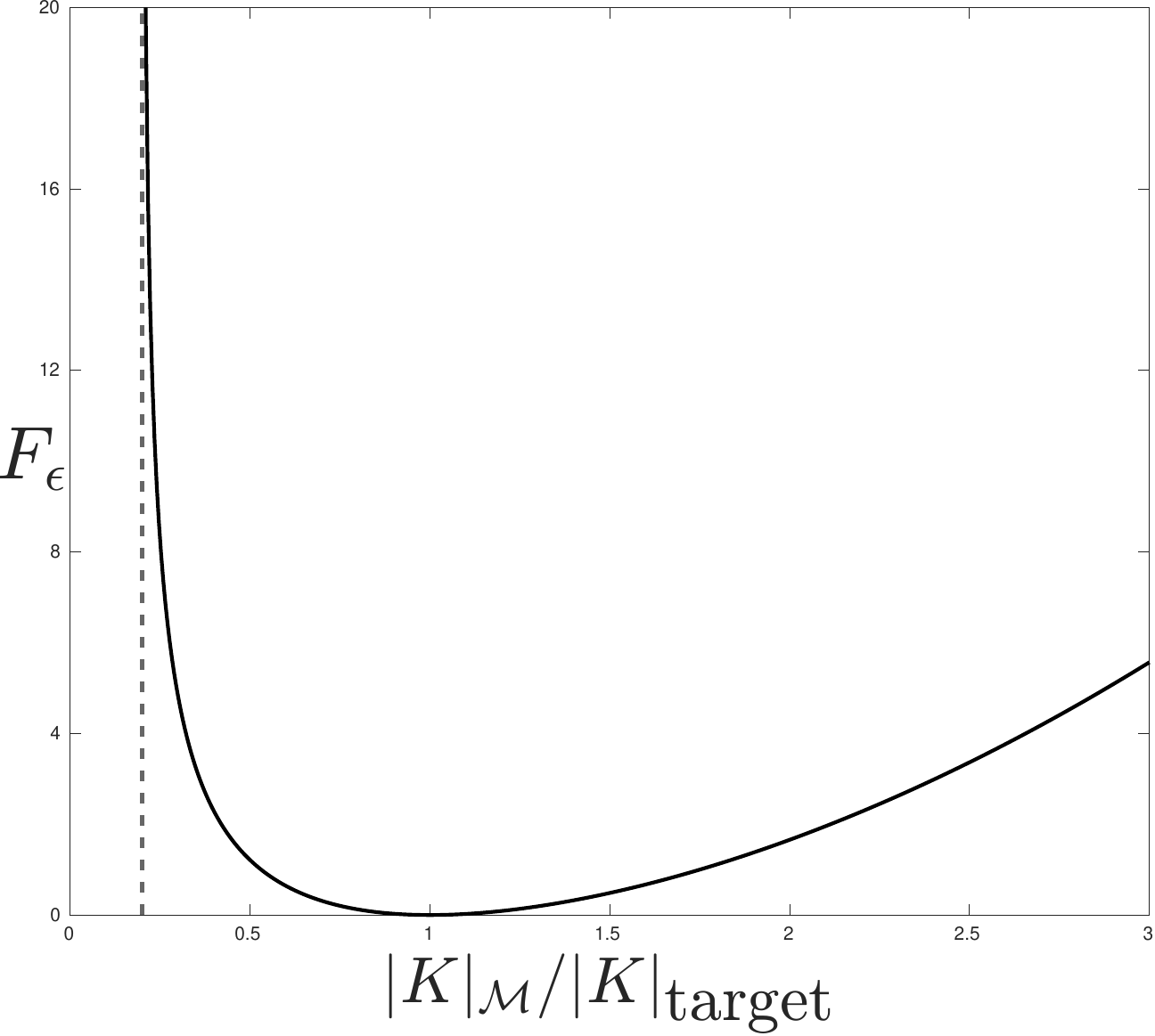}
  \caption{Log-barrier function for $\epsilon = 0.2$.}
  \label{fig:logbarrier}
\end{figure}

The overall cost function $E = \ecal + \fcal$ is smooth if the metric field is smooth,
thus its gradient $\nabla E = \partial E/\partial \bfx_i$ with respect to the position of the vertices of the subtriangles can be evaluated exactly.
However, the cost function is nonlinear and nonconvex,
In addition, it may have infinitely many global minima, as any rotation of the chosen ideal tiling $\tcal_0$ yields a unit mesh.
To minimize $E$, we use the quasi-Newton L-BFGS algorithm implemented in the minimization library \textsc{Ceres} \cite{agarwal2012ceres},
providing it with the exact gradient $\nabla E$.
For each of the examples hereafter, we minimize $E$ over the positions $\bfx_i$ of the vertices of the subtriangles.
The minimization terminates when either the relative decrease in the cost function, the maximum norm of its gradient or the change in the problem parameters $\Vert\Delta \bfx_i\Vert$
falls below the default thresholds of \textsc{Ceres}, which are $10^{-6}, 10^{-10}$ and $10^{-8}$ respectively.


\subsection{Setup and numerical results}
Starting from a uniform square mesh, we divide each triangle, called macrotriangles,
into $N$ subtriangles, which are then used to locally enforce the isometry.
The vertices are moved by minimizing the cost function $E$ at fixed connectivity ($r$-adaptation),
with the objective of obtaining a final mesh isometric to one of the reference tilings $\trieq$ or $\trirec$,
with respect to a given Riemannian metric.
Since the topology is unchanged during the optimization, the starting mesh has the same connectivity
as a regular tiling, with six triangles per interior vertex.
When the metric is nonuniform,
the macrotriangles after minimization have piecewise linear edges that approximate curved edges.
We consider metrics induced by surfaces embedded in $\Rtrois$
that are the graph of a real-valued function $f(x,y)$ and are described by $S = (x,y,f(x,y)) \in \Rtrois$.
The global coordinates $(x,y)$ lie in the so-called parameter space $U \subseteq \R$,
and the component matrix of the induced metric on the surface $S$ reads in these coordinates:
\begin{equation} \label{eq:induced_metric}
[\mcal] = \begin{pmatrix}1 + (\partial_xf)^2 & \partial_xf\partial_yf\\\partial_xf\partial_yf & 1 + (\partial_yf)^2\end{pmatrix}.
\end{equation}
The Riemannian metric on $S$ is the ambient Euclidean metric of $\Rtrois$,
but in the parameter space, it exhibits anisotropy induced by the shape of the surface,
which is visible on the optimized meshes.
These meshes appear as the projection on the parameter space of a piece of fabric laid upon each surface.

To assess the quality of the optimized triangulations,
we compare the edge lengths with those of $K_0$,
and measure the elementwise quality $Q_{K_0}(k) \in [0,1]$
given by \eqref{eq:elementwise_distortion_based_quality},
which is 1 on perfectly isometric triangulations.
For each choice of Riemannian metric,
the approximated geodesics connecting the mesh vertices are superimposed to the curved edges for comparison.
These are obtained by integrating the geodesic equation \cite{lee2018introduction} with an explicit fourth-order Runge-Kutta scheme,
using a bisection to determine the initial velocity required to reach the prescribed endpoint.

\subsubsection{Developable surfaces}
We first consider developable surfaces embedded in $\Rtrois$.
As discussed in \Cref{sec:existence}, these 2-manifolds are isometric to the Euclidean space and can be covered by triangulations that are isometric to
any tiling of the Euclidean space.
We consider the three surfaces $S_i = (x,y,f_i(x))$ obtained by extruding the
graph of $f_1 = x, f_2 = x^2$ and $f_3 = \exp(-10x^2)$ along the $y$-axis,
\Cref{fig:developable_surfaces},
with Riemannian metric:
\begin{equation}
[\mcal_i] = \begin{pmatrix}1 + (f_i'(x))^2 & 0\\0 & 1\end{pmatrix}.
\end{equation}

\begin{figure}[tbhp]
  \centering
  \includegraphics[width=0.9\linewidth]{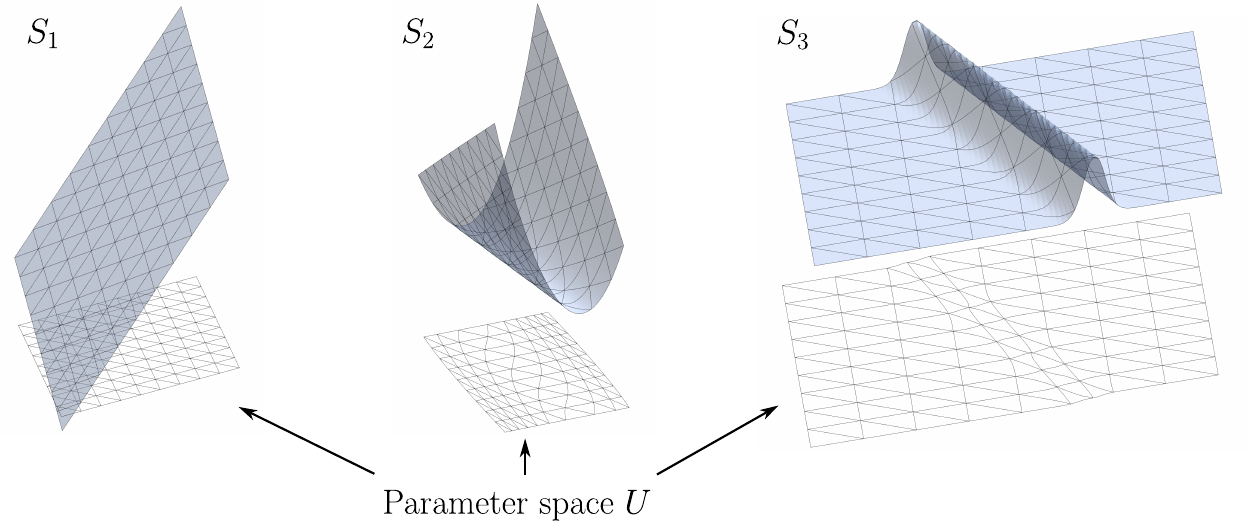}
  \caption{Developable surfaces in $\Rtrois$ and their parameter space $U \subseteq \R$.
  The isometric triangulations are uniform and geodesic on each surface but appear curved and distorted on $U$.}
  \label{fig:developable_surfaces}
\end{figure}
On optimal triangulations, the cost function $E$ is exactly zero,
and any two of these triangulations only differ by the initial rotation of the regular tiling.
There is thus an infinite number of global minima for $E$, as each rotation yields an isometric mesh.
This rotation can be imposed by fixing an edge or subedge of the initial mesh,
but it requires prior knowledge the location of a straight edge.
In our tests, all vertices are free to move, so that the rotation is determined by the initial conditions given to the optimizer.
In particular, different rotations of the initial mesh yield unit triangulations that differ by the rotation applied to the reference tiling before applying the map $F$.

The optimized triangulations are shown in \Cref{fig:unitMeshes_developable_eq,fig:unitMeshes_developable_rec},
and are optimized with respect to $\trieq$ and $\trirec$ respectively.
These triangulations are everywhere isometric to their reference tiling,
and are thus unit curvilinear meshes with respect to $\mcal_i$
in the sense of all three Definitions \ref{def:linear_unit}, \ref{def:unit_curved_triangle_jacobian} and \ref{def:unit_isometry}.
For each surface, optimal triangulations $\tri_j$ obtained from three different rotations $R_j$ of the initial mesh are shown.
They are not rotations of one another, as each of them satisfies either $\tri_j = F(R_j\trieq)$ or $\tri_j = F(R_j\trirec)$,
where $F$ is an isometry between $\R$ and the considered surface.
Instead, they are related by $\tri_2 = F(R_2R_1^TF^{-1}(\tri_1))$.

The metric associated with $S_1$ is constant,
so the unit mesh is linear, as observed at the top of \Cref{fig:unitMeshes_developable_eq,fig:unitMeshes_developable_rec}.
For this particular surface, the isometry $F$ can be written explicitly as $F : \R \to U : \bfy \mapsto \mcalmud R \bfy$ for some rotation matrix $R$,
whereas it requires solving a nonlinear ODE in $f$ in general.
The lengths and areas with respect to this constant metric are computed exactly by quadratures,
so the cost function can be minimized close to machine precision ($10^{-14}$ and $10^{-12}$ with respect to $\trieq$ and $\trirec$ respectively).
Similarly, the metric of $S_3$ becomes the identity away from 0 as $\exp \to 0$ and the mesh quickly becomes linear.
Whenever the metric varies, the isometric triangulations are curved: they can be viewed as straight grids drawn on the surfaces,
then projected onto the parameter space.

The geodesics connecting the mesh vertices are drawn in blue and superimposed on the curved edges.
They perfectly match the edges for all triangulations,
in agreement with Proposition \ref{prop:properties_isometric_elements} as the elements are geodesic triangles.
The edges are the intersections of three geodesic grids that are parallel transported across the surface:
this is straightforward on linear meshes, but can be appreciated on the meshes of $S_2$ and $S_3$.
In these grids, the angles (measured with respect to the local metric) at which two sets of parallel geodesics intersect are either $\pi/3$ when
the reference tiling is $\trieq$ or $\pi/2$ and $\pi/4$ when the tiling is isometric to $\trirec$.
Following Proposition \ref{prop:properties_isometric_elements}, the distortion (not shown) is minimal and equal to 1 on all linear subtriangles,
and the curved edges are isometric to the edges of the reference triangle, \cref{fig:edgeLength_developable}.
Together, these two criteria ensure that the isometries are enforced globally (edges lengths) and locally (distortion + edges lengths).

\begin{figure}[tbhp]
  \centering
  \begin{tikzpicture}
    \node[anchor=south west,inner sep=0] (image) at (0,0)
    {\includegraphics[width=0.9\linewidth]{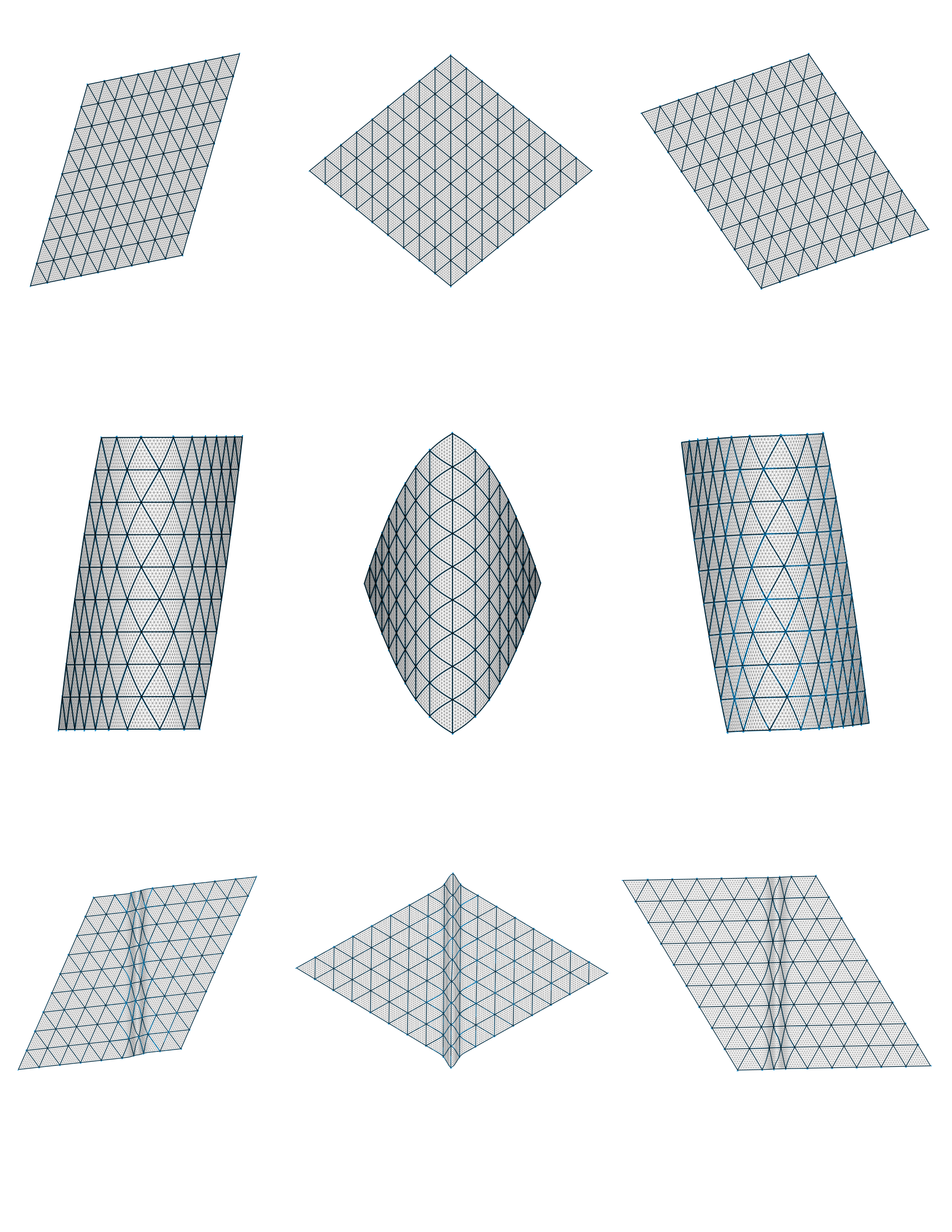}};
    \begin{scope}[x={(image.south east)},y={(image.north west)}]
      \node at (   0,0.85) {$x$};
      \node at (0.15,0.73) {$E_\text{initial} = \texttt{3.9862e+00}$};
      \node at (0.45,0.73) {$E_\text{initial} = \texttt{5.8737e+00}$};
      \node at (0.85,0.73) {$E_\text{initial} = \texttt{7.8995e+00}$};
      \node at (0.15,0.70) {$E_\text{final} ~=  \texttt{1.3417e-14}$};
      \node at (0.45,0.70) {$E_\text{final} ~=  \texttt{7.2833e-14}$};
      \node at (0.85,0.70) {$E_\text{final} ~=  \texttt{1.4437e-14}$};
      \node at (   0,0.55) {$x^2$};
      \node at (0.15,0.355) {$E_\text{initial} = \texttt{3.9044e+00}$};
      \node at (0.45,0.355) {$E_\text{initial} = \texttt{3.4381e+03}$};
      \node at (0.85,0.355) {$E_\text{initial} = \texttt{4.6694e+03}$};
      \node at (0.15,0.325) {$E_\text{final} ~=  \texttt{3.0017e-08}$};
      \node at (0.45,0.325) {$E_\text{final} ~=  \texttt{1.5992e-03}$};
      \node at (0.85,0.325) {$E_\text{final} ~=  \texttt{2.1175e-05}$};
      \node at (0,0.22) {$e^{-10x^2}$};
      \node at (0.15,0.09) {$E_\text{initial} = \texttt{3.7692e+00}$};
      \node at (0.45,0.09) {$E_\text{initial} = \texttt{7.7535e+00}$};
      \node at (0.85,0.09) {$E_\text{initial} = \texttt{9.0946e+00}$};
      \node at (0.15,0.06) {$E_\text{final} ~=  \texttt{6.2525e-05}$};
      \node at (0.45,0.06) {$E_\text{final} ~=  \texttt{2.0653e-03}$};
      \node at (0.85,0.06) {$E_\text{final} ~=  \texttt{1.2795e-04}$};
    \end{scope}
  \end{tikzpicture}
  \caption{Triangulations with $N = 10$ subdivisions which are isometric to $\trieq$.}
  \label{fig:unitMeshes_developable_eq}
\end{figure}

\begin{figure}[tbhp]
  \centering
  \includegraphics[width=0.9\linewidth]{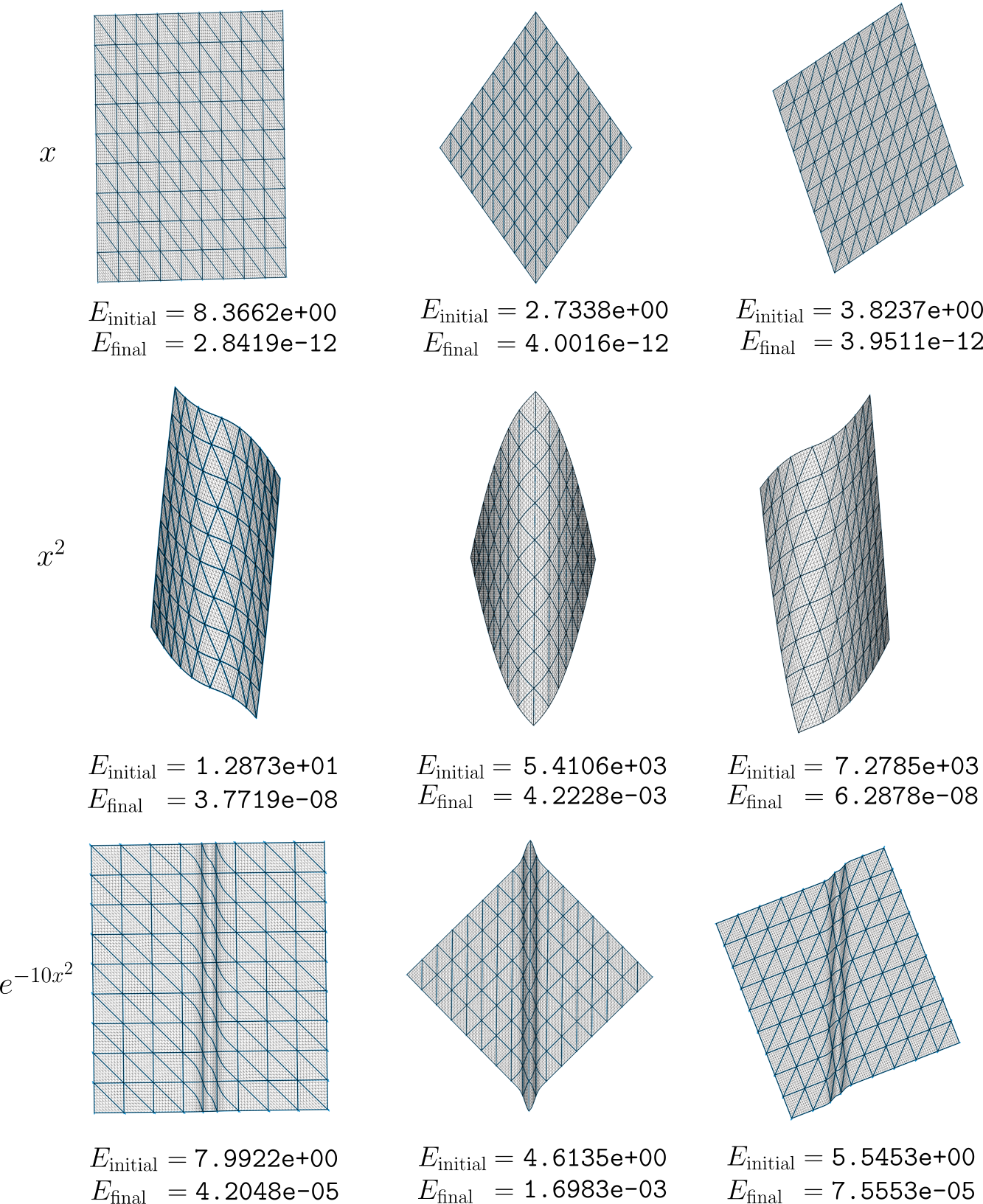}
  \caption{Triangulations with $N = 8$ subdivisions which are isometric to $\trirec$.}
  \label{fig:unitMeshes_developable_rec}
\end{figure}

\begin{figure}[tbhp]
  \includegraphics[width=\linewidth]{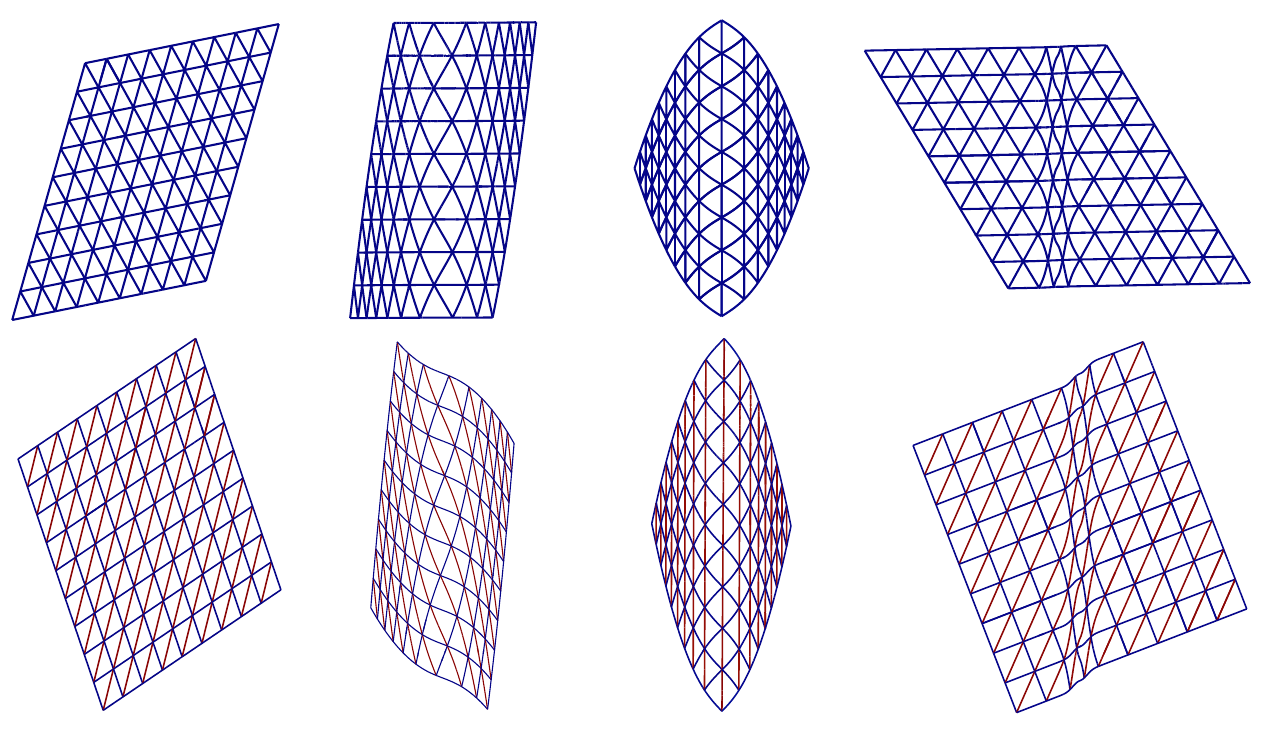}
  \caption{Edge length with respect to the metric. On top, all edges are unit. On the bottom, the edges are either unit (blue) or of length $\sqrt{2}$ (red).}
  \label{fig:edgeLength_developable}
\end{figure}

\hfill\break
\paragraph{Isometric $\pdeux$ triangles}
In addition to being isometric, the optimized triangulations are also quadratic (i.e., consist only of $\pdeux$ triangles) if $\mcalmud$ is linear.
The metric
\begin{equation}
  \mcalmud =
  \begin{pmatrix}
    x & 0\\
    0 & 1
  \end{pmatrix}
  ~~\rightarrow~~
  [\mcal] =
  \begin{pmatrix}
    1 + (f'(x))^2 & 0\\
    0 & 1
  \end{pmatrix}
  =
  \begin{pmatrix}
    x^{-2} & 0\\
    0 & 1
  \end{pmatrix}
\end{equation}
is induced by the surface $S_4 = (x,y,f_4(x))$, with $f_4(x)$ defined on $(0,1)$ by:
\begin{equation}
f_4(x) = \frac{x \sqrt{x^{-2}-1} \left( \sqrt{x^2-1} - \arctan(\sqrt{x^2-1}) \right)}{\sqrt{x^2-1}},
\end{equation}
so that $f_4'(x) = \sqrt{x^{-2}-1}$.
This function is real-valued on $(0,1)$ even though it involves the roots of $x^2-1$.
To avoid complications outside of $(0,1)$ where $f$ diverges or is complex-valued,
the initial mesh covers $[0.0174, 0.3495] \times [-2,2]$ when optimizing for $\trieq$ and $[0.0174, 0.95] \times [-2,2]$ for $\trirec$.
This yields meshes with respectively 3 and 4 unit edges along $x$ and 4 along $y$.
When optimizing for $\trieq$, the vertices can move freely, whereas for $\trirec$ they are constrained to slide along the boundary,
as the initial mesh is close to the optimal triangulation.
In that case, the boundary vertices of the initial mesh are obtained by bisection to split the upper and lower boundaries into 4 unit edges.
The initial subdivision is uniform and the additional vertices on the boundaries are free to slide to comply with the metric.
The optimized triangulations are shown in \cref{fig:isoP2}:
because $\mcalmud$ is linear,
the resulting triangles are parabolic and isometric to either $\keq$ or $\krec$,
that is, they are perfectly unit $\pdeux$ triangles.
%

\begin{figure}[tbhp]
  \includegraphics[width=\linewidth]{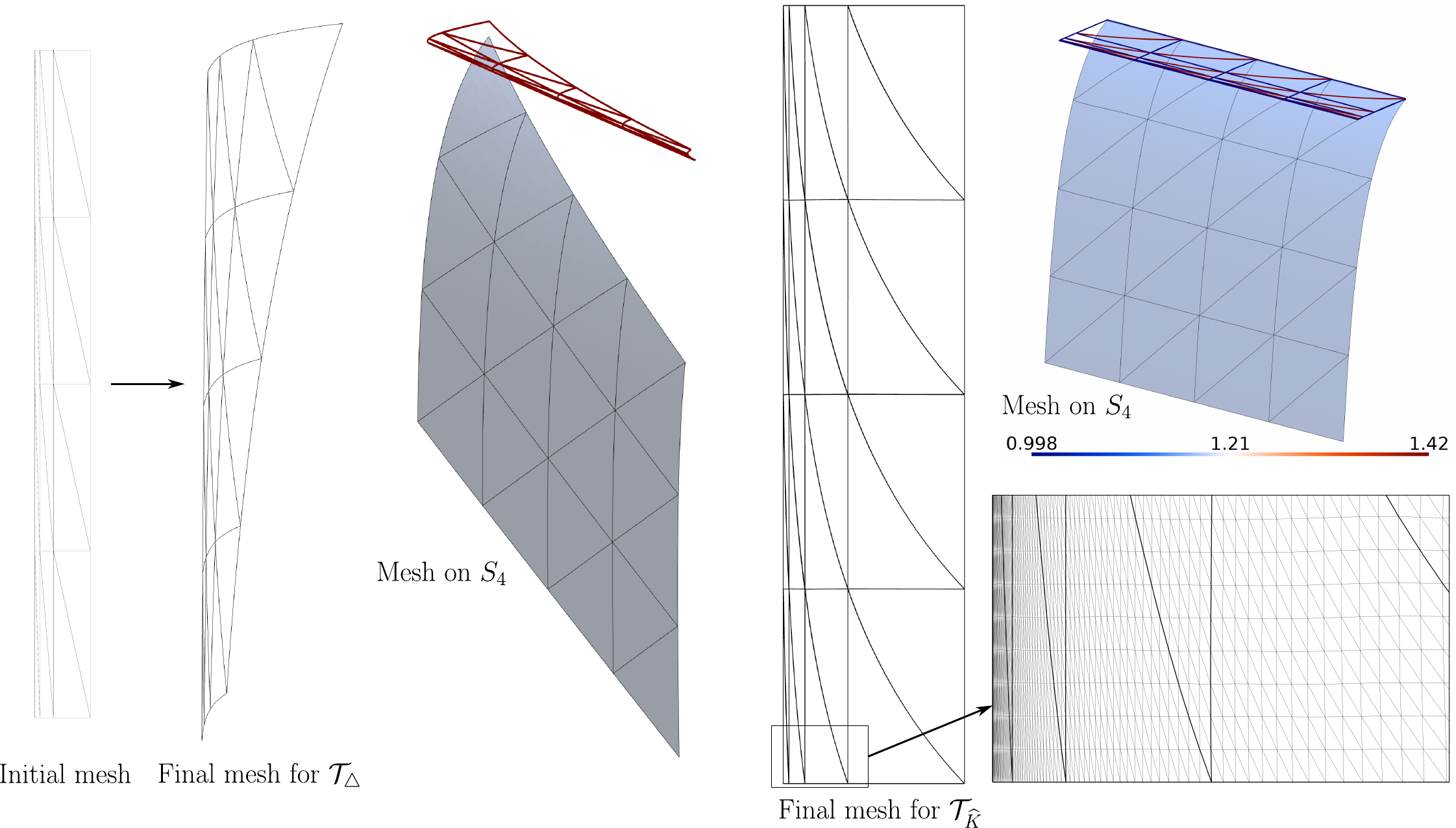}
  \caption{Isometric $\pdeux$ triangulations with respect to $S_4$. The equilateral mesh has unit edges everywhere (red),
  and the mesh made of right triangles has unit edges (blue) and edges of length $\sqrt{2}$ (red).}
  \label{fig:isoP2}
\end{figure}

\subsubsection{Non-developable surfaces}
Next, we consider isometric triangulations on less trivial surfaces.

\paragraph{Radial function}
The surface $S_5$ is the graph of $f(x,y) = x^2 +y^2 = r^2$.
In polar coordinates, the metric writes:
\begin{equation}
  [\mcal] =
  \begin{pmatrix}
    1 + 4r^2 & 0\\
    0 & 1
  \end{pmatrix},
\end{equation}
hence the length of an arc on the surface with constant radius is equal to its Euclidean length,
and the metric only modifies the computation of lengths in the radial direction.

Away from the origin, we can expect the optimized mesh to be very close to isometric.
In \cref{fig:x2y2_shifted}, we start from a uniform mesh of $[-1,5] \times [-3,3]$, not centered at the origin.
The optimized meshes are essentially isometric to $\trieq$ or $\trirec$:
the curved edges match the geodesics, the distortion-based quality lies in $[0.992, 1]$ for all subtriangles
and the edges preserve the length of the reference simplex with respect to the metric.
The edge lengths lie within $[0.94,1.02]$ for triangles isometric to the unit equilateral triangle
and within $[0.96,1.42]$ for triangles isometric to the right triangle of sides 1 and $\sqrt{2}$.
The honeycomb pattern of $\trieq$ and the split squares pattern of $\trirec$ differ only by their inner angles with respect to the metric,
which are $\pi/3$ for the former and $\pi/2$ and $\pi/4$ for the latter
(the exact angles of the optimized meshes with respect to the metric are not shown).
In the Euclidean space, it is thus subtle to distinguish between both patterns,
as their inner angles depends on the position,
however, the edge lengths plot shows which edges are longer than the others and gives a better idea of the inner angles.
The optimized triangulation for $\trieq$ is projected back onto $S_5$ in \cref{fig:x2y2_3d},
together with two optimized triangulations centered at the origin and not isometric to $\trieq$.
\cref{fig:x2y2_3d} provides a quick visual assessment the quality of the optimization:
geodesic meshes should appear straight on the surface despite being curved in the parameter space.
This is clearly not the case for the first two meshes;
however, the last one, which is farther from the origin, is very close to being isometric,
and its projection on the surface is a uniform tiling of straight equilateral triangles.

\begin{figure}[tbhp]
  \centering
  \begin{tikzpicture}
    \node[anchor=south west,inner sep=0] (image) at (0,0)
    {\includegraphics[width=0.9\linewidth]{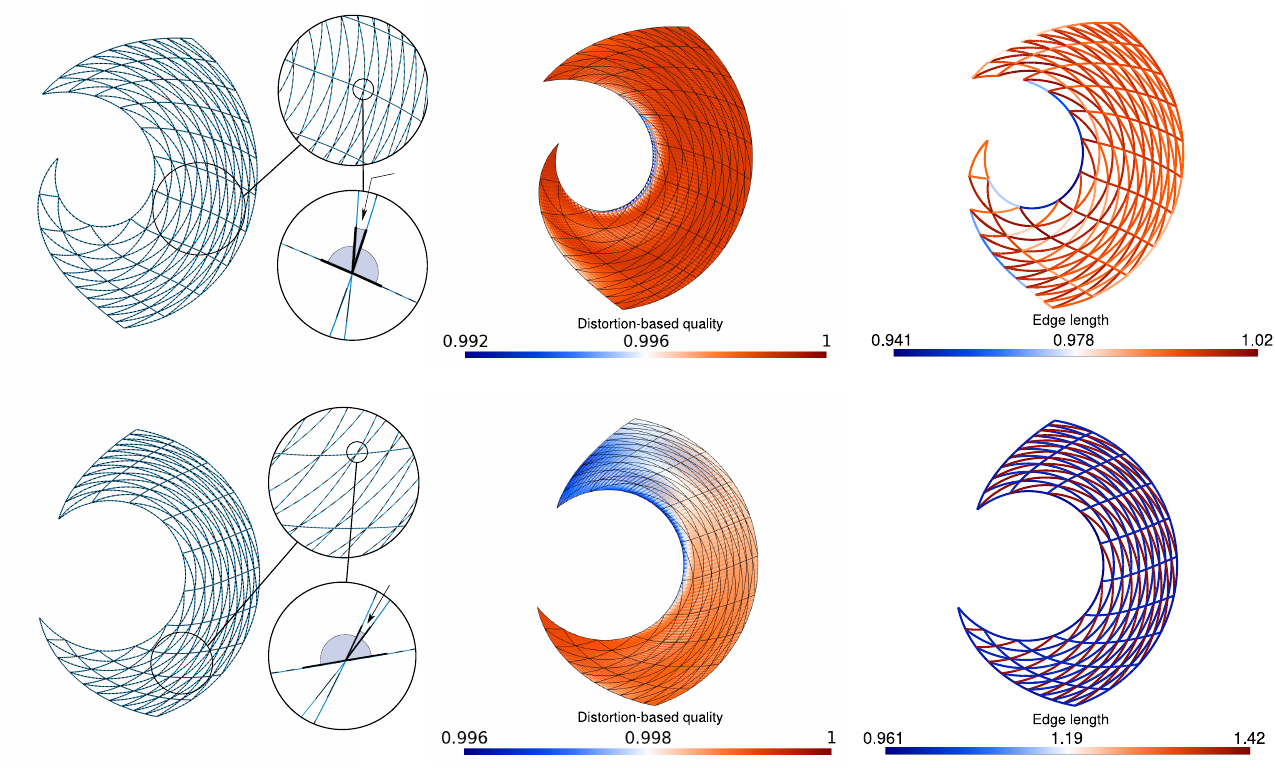}};
    \begin{scope}[x={(image.south east)},y={(image.north west)}]
      \node at (0,0.76) {$\trieq$};
      \node at (0,0.26) {$\trirec$};

      \node at (0.26,0.71) {$\frac{\pi}{3}$};
      \node at (0.325,0.775) {$\frac{\pi}{3}$};
      \node at (0.31,0.67) {$\frac{\pi}{3}$};

      \node at (0.24,0.17) {$\frac{\pi}{2}$};
      \node at (0.315,0.26) {$\frac{\pi}{4}$};
      \node at (0.31,0.18) {$\frac{\pi}{4}$};
    \end{scope}
  \end{tikzpicture}
  \caption{Quasi-isometric triangulations for $S_5$ away from the origin.}
  \label{fig:x2y2_shifted}
\end{figure}
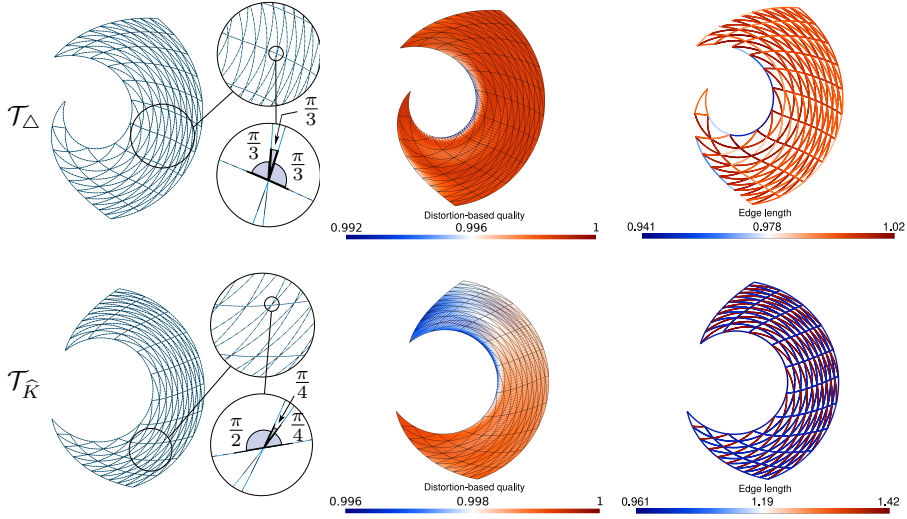

\begin{figure}[tbhp]
  \centering
  \includegraphics[width=\linewidth]{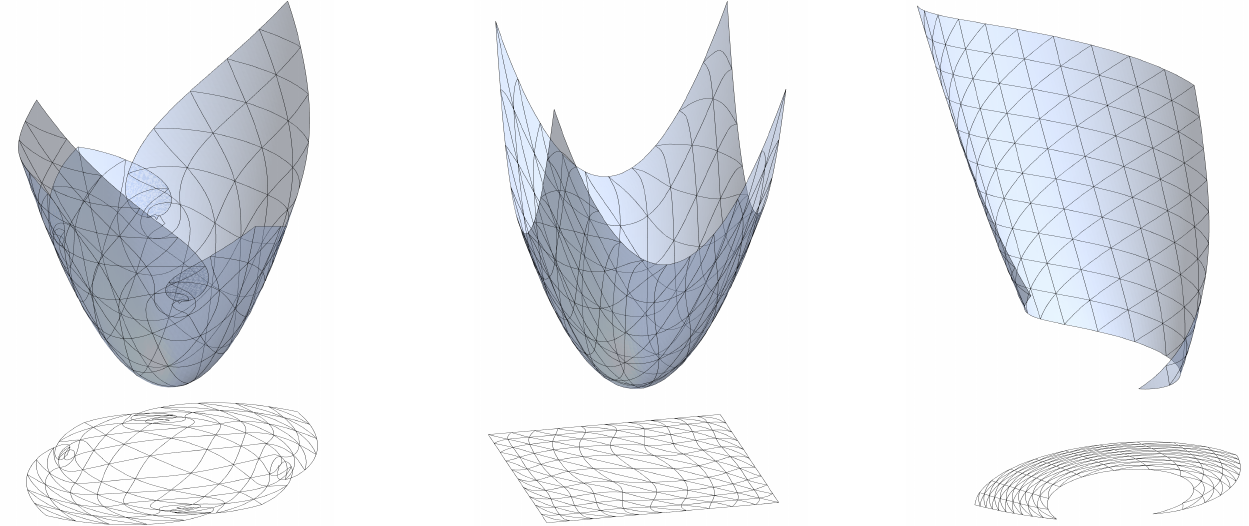}
  \caption{Optimized triangulations for $S_5 = (x,y,x^2+y^2)$ and their projection on $S_5$.}
  \label{fig:x2y2_3d}
\end{figure}

\paragraph{Undulating boundary layer}
Finally, the surface $S_6$ is adapted from the one proposed in \cite{aparicio2022high}.
It mimics an undulating boundary layer, and its metric is defined by:
\begin{equation}
  [\mcal] = J_\varphi^T
  \begin{pmatrix}
    1 & 0\\
    0 & h^{-2}(x,y)
  \end{pmatrix}
  J_\varphi,
  ~~~~\text{with}~~~~
  \varphi(x,y) = \left( x, \frac{10y - \cos(2\pi x)}{\sqrt{100 + 4\pi^2}} \right)
\end{equation}
and where the sizing function is given by:
\begin{equation}
  h(x,y) = 0.1 + 2 \left|\frac{10y - \cos(2\pi x)}{\sqrt{100 + 4\pi^2}}\right|.
\end{equation}
The mesh size is constant in the $x$-direction (before applying the map $\varphi$)
and grows from $0.1$ in the $y$-direction according to $h(x,y)$.
The anisotropic quotient of $\mcal$, defined in two dimensions by \cite{loseille2008adaptation}:
\begin{equation}
  \text{quo} = \max_{i\,=\,1,2} \sqrt{\frac{h_i^2}{h_1h_2}} = \max_{i\,=\,1,2} \frac{(\det \mcal)^{1/4}}{\sqrt{\lambda_i}}
\end{equation}
is shown in \cref{fig:anisotropic_quotient}.
It is a ratio of principal lengths obtained from the ratio of areas associated to the principal sizes $h_1, h_2$.
The highest anisotropy occurs along the curve $10y - \cos(2\pi x) = 0$.
\begin{figure}[hbt!]
  \centering
  \includegraphics[width=\linewidth]{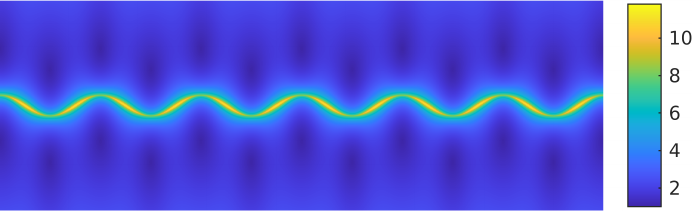}
  \caption{Anisotropic quotient of $\mcal$ for $S_6$ in log scale for $[-3,3] \times [-1,1]$.}
  \label{fig:anisotropic_quotient}
\end{figure}
The initial triangulation is a uniform mesh of $[-1,1]^2$ and all vertices are free to move.
The optimized meshes with respect to $\trieq$ and $\trirec$ are shown in \cref{fig:wave16rec}.
The curved elements follow the waves of the boundary layer as expected.
Because the characteristic mesh size is set to 1 along the $x$-direction and the argument of the cosine is $2\pi x$,
a unit edge covers exactly one wavelength.
Thus, to accurately represent the curved edges with high-order triangles,
polynomials of degree at least 3 are required.
For both optimized meshes, the quality of the subtriangles lies within $[0.9, 1]$ almost everywhere, but drops as low as 0.72
in the boundary layer where the curvature is the highest.
This drop depends directly on the number of subtriangles used to represent the region of high curvature
and can be reduced by further refining the macrotriangles.
The edge lengths with respect to the metric agree with the chosen reference triangle,
and homologous edges are quasi-isometric.
In \cite{aparicio2022high},
the meshes are optimized with respect to the distortion $\eta_{\keq}$.
As the distortion is length-agnostic and yields conformally flat triangulations,
the elements size is determined by the initial mesh.
In other words, the meshes obtained in \cite{aparicio2022high} after distortion minimization are (quasi-)conformally equivalent
to $\trieq$, that is, they preserve the angle of the tiling.
They are not (quasi-)isometric to $\trieq$ in general, unless
the optimized mesh also happens to preserve the lengths.
Here, the length is enforced locally instead and the minimization of the distortion is achieved
as a consequence of the macrotriangles being quasi-isometric to either $\keq$ or $\krec$,
in agreement with Proposition \ref{prop:properties_isometric_elements}.

\begin{figure}[tbhp]
  \includegraphics[width=\linewidth]{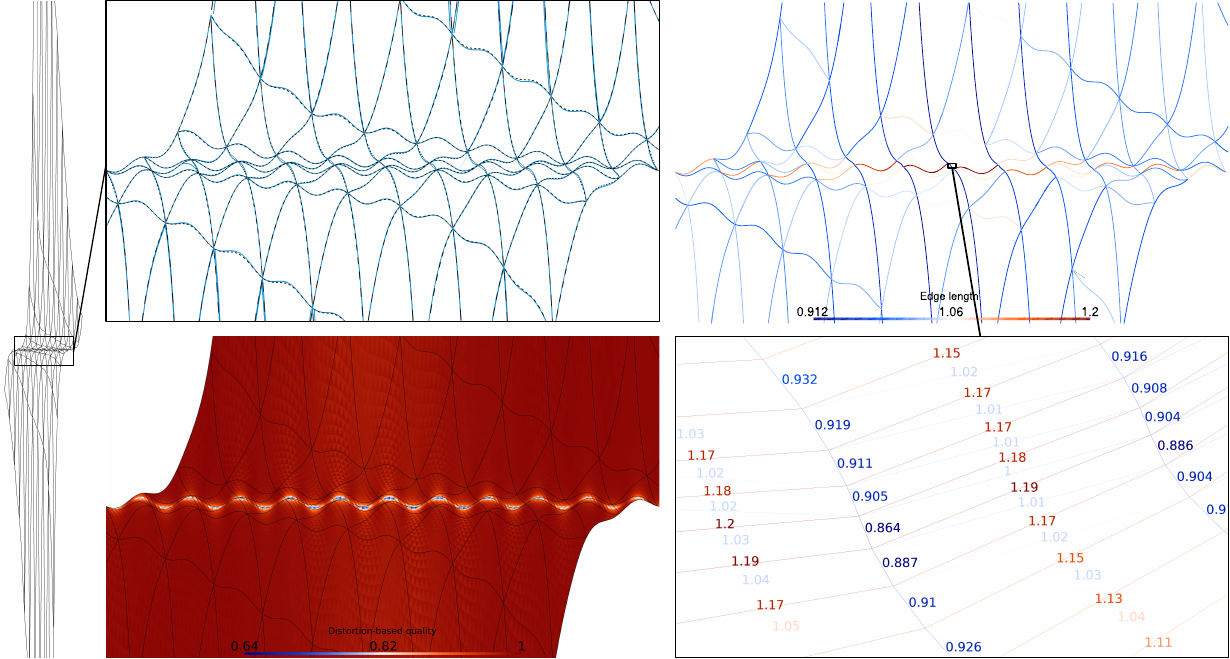}
%
  \includegraphics[width=\linewidth]{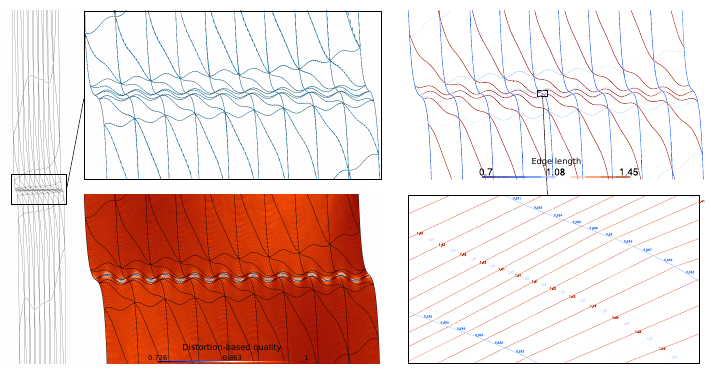}
  \caption[Isometric triangulation for the undulating boundary layer]{Triangulations for $S_6$
  optimized with respect to $\trieq$ (top, $N = 12$) and $\trirec$ (bottom, $N = 16$).
  First and third rows: curvilinear mesh with geodesics and edges length with respect to the metric.
  Second and fourth rows: distortion-based quality and edge length of some subtriangles.}
  \label{fig:wave16rec}
\end{figure}

\section{Conclusions}
\label{sec:conclusions}
The generalization to high-order of the notion of unit elements, a central component of the continuous mesh framework,
was the main focus of this paper.
Existing definitions of unit simplices were reviewed, and their limitations for high-order mesh generation were highlighted.
A new definition based on Riemannian isometries was introduced, encompassing existing ones by defining unit elements as isometries of the regular simplex $\keq$.
Isometric unit simplices are curved in general, but they only exist in the restrictive case of flat manifolds.
Proofs of concept of isometric triangulations were obtained in this idealized setting by optimizing a mesh at fixed connectivity.
For practical applications, the notion of quasi-unitness was extended to enable the generation of quasi-isometric simplices with respect to arbitrary metrics.

While the adequacy between a metric and a high-order mesh was addressed, the adequacy between solution and metric, that is, metric-based error estimates for high-order discretizations, remains to be treated.
Indeed, a complete methodology for high-order mesh adaptation requires both accurate error estimates translated into a Riemannian metric,
and a way of generating unit triangulations for this metric.
Interpolation error estimates on high-order meshes were already discussed in \cite{rocheryThesis, bawin2024metric},
and shall be the object of future work.

\appendix

\section*{Acknowledgments}
The authors would like to thank the Belgian Fund for Scientific Research (FRS-FNRS/Grant FRIA/FC 29571) for their support.
Financial support from the Simulation-based Engineering Science (Génie Par la Simulation) program funded through the CREATE program of
the Natural Sciences and Engineering Research Council (NSERC) of Canada is also gratefully acknowledged.

\bibliographystyle{siamplain}
\bibliography{references}
\end{document}